\documentclass[11pt]{amsart}
\usepackage[utf8]{inputenc}
\usepackage[doi=false, url=false,isbn=false,style = alphabetic]{biblatex}
\usepackage[margin = 1.25in]{geometry}
\usepackage{amsmath}
\usepackage{amssymb}
\usepackage{amsthm}
\usepackage{hyperref}
\usepackage{amsfonts}
\usepackage{enumitem}
\usepackage{tikz}
\usepackage{tikz-cd}
\usepackage{tabularx}
\usepackage{mathtools}
\usepackage{url}
\usepackage{makecell}
\usepackage{tikz-3dplot}
\usepackage{newtxtext, newtxmath}

\tikzcdset{scale cd/.style={every label/.append style={scale=#1},
    cells={nodes={scale=#1}}}}

\newtheorem{thm}{Theorem}[section]
\newtheorem*{thm*}{Theorem}
\newtheorem{cor}[thm]{Corollary}
\newtheorem{lem}[thm]{Lemma}

\newtheorem*{claim*}{Claim}
\newtheorem{prop}[thm]{Proposition}
\newtheorem{question}[thm]{Question}
\newtheorem{conj}[thm]{Conjecture}
\theoremstyle{definition}
\newtheorem{example}[thm]{Example}
\newtheorem{rmk}[thm]{Remark}
\newtheorem{definition}[thm]{Definition}

\newcommand{\lex}{\mathrm{lex}}
\newcommand{\exc}{\mathrm{exc}}
\newcommand{\HPoin}{\mathrm{HPoin}}
\renewcommand{\H}{\mathrm{H}}
\DeclareMathOperator{\coker}{coker}
\DeclareMathOperator{\im}{im}
\DeclareMathOperator{\Gr}{Gr} 
\DeclareMathOperator{\Star}{St}
\DeclareMathOperator{\Sym}{Sym}
\DeclareMathOperator{\init}{in}
\DeclareMathOperator{\pos}{pos}
\DeclareMathOperator{\spann}{span}
\DeclareMathOperator{\rk}{rk}
\DeclareMathOperator{\rks}{rks}
\DeclareMathOperator{\Tor}{Tor}

\title[The singular cohomology ring of a uniform matroid]{The singular cohomology ring of a uniform matroid: \\ Combinatorics and Lefschetz properties}
\author{Kyle Binder}
\address{Department of Mathematics, Louisiana State University}
\email{kbinde1@lsu.edu}
\begin{document}

\begin{abstract}
    The singular cohomology ring of a matroid is an algebraic invariant which generalizes the Chow ring of a matroid.
    We study combinatorial and Lefschetz properties of the singular cohomology ring of a uniform matroid.
    Combinatorially, we construct an explicit basis for the singular cohomology ring in terms of 
    Koszul homology. From this basis we derive multiple formulas for the Hodge numbers of the cohomology
    ring that recover and extend known formulas for the Chow polynomial of a uniform matroid.
    We also use this basis to show that the singular cohomology ring of a uniform matroid satisfies the 
    \emph{quasi-projective strong Lefschetz property}---a slight weakening of the Hard Lefschetz property
    found in the Chow ring of a matroid.
\end{abstract}

\maketitle

\section{Introduction}\label{section:introduction}
    The \emph{singular cohomology ring} of a loopless matroid $M$ was introduced in \cite{binder} as an extension of 
    the well-studied \emph{Chow ring} of a matroid. Both algebraic invariants are defined via the smooth toric
    variety $ X_{\Sigma_{M}} $ associated to the \emph{Bergman fan} $ \Sigma_{M} $ of the matroid.
    The Chow ring of $ M $ is the Chow ring $ A^{\bullet}(X_{\Sigma_{M}}) $ of the toric variety, while
    the singular cohomology ring of $M$ is the singular cohomology ring $H^{\bullet}(X_{\Sigma_{M}}) $ of the toric variety.
    (All Chow and singular cohomology rings here are taken with rational coefficients.)
    As the cycle class map $ A^{\bullet}(X_{\Sigma_{M}}) \hookrightarrow H^{2\bullet}(X_{\Sigma_{M}}) $
    is injective for smooth toric varieties \cite{totaro}, the singular cohomology ring of $M$ contains the Chow ring as a subring.
    
    At this stage there is still much to be discovered about these singular cohomology rings. In general, we have vanishing results
    for the cohomology and a computation of the top-weight cohomology \cite[Theorem 6.6]{binder}, but we do not have formulas
    for the Betti numbers or even a description for generators of the singular cohomology ring. The algebraic structure of the singular cohomology
    ring is also not understood. One of the salient properties of the Chow ring of a matroid is that it satisfies the \emph{K\"{a}hler package},
    and in particular, has the Hard Lefschetz property for multiplication by a generic linear form. It is not known whether the singular
    cohomology ring of a matroid also has such properties. 

    The goal of this paper is to address these questions in the context of uniform
    matroids. Here, we construct an explicit, combinatorial basis for the singular cohomology ring of a uniform matroid in terms of Koszul homology, 
    use these bases to produce explicit formulas for the Betti numbers (and the finer \emph{Hodge numbers}), and show that, even though
    Hard Lefschetz fails, the 
    singular cohomology ring of a uniform matroid satisfies the \emph{quasi-projective strong Lefschetz property}.

    Our results also provide useful information about the Chow rings of uniform matroids. For example, the basis we construct for
    the singular cohomology ring of a uniform matroid restricts to a new and combinatorially rich basis for the Chow ring.
    By studying the combinatorics of the basis elements, we produce combinatorial proofs of several formulas for the
    Chow polynomials of uniform matroids due to  Hoster \cite{hoster}, Hameister, Rao, and Simpson \cite{HRS}, and 
    Ferroni, Matthews, Stevens, and Vecchi \cite{FMSV}. This suggests that the singular cohomology ring may be a useful tool in studying the 
    Chow ring of a general matroid.

    Let us now describe our main results in more detail.

    \subsection{Combinatorics of the singular cohomology ring of a uniform matroid}\label{subsection:introductionCombinatorics}
    The Chow ring of a loopless matroid on the ground set $[n] = \left\{1, \dots, n \right\} $
    has a well-known combinatorial presentation given by Feichtner and Yuzvinsky \cite{FY}:
    \begin{equation}
        A^{\bullet}(X_{\Sigma_{M}}) = \frac{\mathbb{Q}[\Sigma_{M}]}{\left( \sum_{i \in F} x_{F} - 
            \sum_{j \in G} x_{G} : i,j \in [n] \right).}
    \end{equation}
    Here, $ \mathbb{Q}[\Sigma_{M}] $ is the \emph{Stanley--Reisner ring} of the Bergman fan. This is the ring generated by
    indeterminates $ x_{F} $ indexed by the proper, non-empty flats of $M$ and subject only to the relation $ x_{F} x_{G} = 0 $ 
    whenever $ F $ and $G $ are incomparable. Feichtner and Yuzvinsky also produce a basis for the Chow ring using Gr\"{o}bner
    bases.\footnote{In their paper, Feichtner and Yuzvinsky use a slightly different presentation than what we use here. They
    include an indeterminate for the maximal flat and quotient by the linear forms $ \sum_{i \in F} x_{F} $.}
    
    The singular cohomology ring similarly has a combinatorial description. Let $ N_{\mathbb{Q}}^{\vee} $ be the character
    lattice of $ X_{\Sigma_{M}} $ tensored with $ \mathbb{Q} $. This vector space is spanned by elements $ \ell_{i} - \ell_{j} $ for
    $i,j \in [n] $. We then have the bigraded  \emph{Koszul complex}
    \[ 
        K_{i}(\mathbb{Q}[\Sigma_{M}])_{j} = \mathbb{Q}[\Sigma_{M}]_{j-i} \otimes \bigwedge^{i} N_{\mathbb{Q}}^{\vee}
    \]
    whose differential satisfies the Leibniz rule, with $ d(x_{F}) = 0 $ and $ d(\ell_{i} - \ell_{j}) = \sum_{i \in F} x_{F} - 
    \sum_{j \in G} x_{G} $. Taking homology produces an isomorphism of rings
    \[ 
        H_{\bullet}(K(\mathbb{Q}[\Sigma_{M}])) \cong H^{\bullet}(X_{\Sigma_{M}}). 
    \]
    Moreover, the bigrading of the Koszul complex induces the weight filtration on $H^{\bullet}(X_{\Sigma_{M}}) $ \cite{weberWeights}
    \[ 
        H_{i}(K(\mathbb{Q}[\Sigma_{M}]))_{j} \cong \Gr_{2j}^{W} H^{2j-i}(X_{\Sigma_{M}}),
    \]
    and the Chow ring corresponds to $H_{0}(K(\mathbb{Q}[\Sigma_{M}]))$ and thus the associated gradeds 
    $ \Gr_{2j}^{W} H^{2j}(X_{\Sigma_{M}}) $.

    Our first result in this paper is an explicit basis for the singular cohomology of uniform matroids. We describe
    this basis using what we call \emph{retral} and \emph{weakly retral} chains of flats.
    \begin{thm}
		The set
		\[ 
            \left\{ \prod_{F \in \mathscr{F}} x_{F} : \textrm{$\mathscr{F}$ is retral} \right\} \cup
            \left\{ \prod_{F \in \mathscr{F}} x_{F} \otimes \xi : \textrm{$ \mathscr{F} $ is weakly
					retral and $ \xi $ is admissible to $ \mathscr{F} $} \right\}
		\]
		is a basis of $ H_{\bullet}(K(\mathbb{Q}[\Sigma_{U_{r,n}}]))
		\cong H^{\bullet}(X_{\Sigma_{U_{r,n}}}) $ 
		as a $ \mathbb{Q} $-vector space.
	\end{thm}
    \noindent We construct this basis in Section \ref{section:basis} using Gr\"{o}bner theory and an induction on rank. One of the novelties of
    this result is that for $ r < n $ the toric variety $ X_{\Sigma_{U_{r,n}}} $ is not proper. To this point, most work
    on the cohomology of toric varieties has focused on the compact (and potentially singular) case.

    Restricting to $ H_{0}(K(\mathbb{Q}[\Sigma_{U_{r,n}}])) $ produces a new basis for the Chow ring of a uniform matroid.
    \begin{cor}
        The set 
        \[ 
            \left\{ \prod_{F \in \mathscr{F}} x_{F} : \textrm{$\mathscr{F}$ is retral} \right\}
        \]
        is a basis of $A^{\bullet}(X_{\Sigma_{M}}) $ as a $ \mathbb{Q} $-vector space.
    \end{cor}
    
    Our second main result is to produce combinatorial formulas for the \emph{Hodge numbers}---the dimensions of the 
    associated gradeds of the weight filtration---of the singular cohomology of $ U_{r,n}$. This is the subject of Section \ref{section:combinatorics}.

    We will package the Hodge numbers of the singular cohomology of a matroid $M$ in two ways. The first is in terms of a \texttt{Macaulay2}
    Betti diagram\footnote{This is precisely the \texttt{Macaulay2} Betti diagram for $ \mathbb{Q}[\Sigma_{M}] $ 
    as an $ \Sym(N_{\mathbb{Q}}^{\vee}) $-algebra.}
    \vspace*{1em}
    \begin{center}
        \begin{tabular}{c|ccccc}
            & $0$ & $1$ & $2$ & $3$ &  $ \cdots $ \\ \hline
            $0$ & $ \beta_{0,0} $ & $ \beta_{1,1} $ & $ \beta_{2,2} $ & $ \beta_{3,3} $ & $ \cdots $ \\
            $1$ & $ \beta_{0,1} $ & $ \beta_{1,2} $ & $ \beta_{2,3} $ & $ \beta_{3,4} $ & $ \cdots $  \\
            $2$ & $ \beta_{0,2} $ & $ \beta_{1,3} $ & $ \beta_{2,4} $ & $ \beta_{3,5} $ & $ \cdots $ \\
            $ \vdots $ & $ \vdots $ & $ \vdots $ & $ \vdots $ & $ \vdots $ & $ \ddots $
        \end{tabular}
    \end{center} 
    \vspace*{1em}
    where $ \beta_{i,j} = \dim H_{i}(K(\mathbb{Q}[\Sigma_{M}]))_{j}  =
    \dim \Gr_{2j}^{W}H^{2j-i}(X_{\Sigma_{M}}) $. The first column of the Betti diagram displays the Betti numbers of the Chow ring.

    The second packaging of Hodge numbers is with the bivariate \emph{Hodge--Poincar\'{e} polynomial}
    \[ 
        \HPoin_{M}(w,x) = \sum_{i,j} \dim H_{i}(K(\mathbb{Q}[\Sigma_{M}]))_{j} \cdot w^{i} x^{j}
        = \sum_{i,j} \dim \Gr_{2j}^{W}H^{2j-i}(X_{\Sigma_{M}}) \cdot w^{i} x^{j}.
    \]
    It will often be convenient to break the Hodge--Poincar\'{e} polynomial into univariate \emph{refined Hodge--Poincar\'{e} polynomials}
    \[ 
        \underline{\H}^{i}_{M}(x) = \sum_{j} \dim H_{i}(K(\mathbb{Q}[\Sigma_{M}]))_{j} \cdot x^{j-i},
    \]
    whence
    \[ 
        \HPoin_{M}(w,x) = \sum_{i=0}^{n-r} \underline{\H}^{i}_{M}(x)\cdot w^{i}x^{i}.
    \]
    Note that $ \underline{\H}^{i}_{M}(x) $ is the generating function for the $i$th column of the Betti diagram
    for the singular cohomology ring and that $ \underline{\H}^{0}_{M}(x) $ is the \emph{Chow polynomial} 
    (or \emph{Hilbert--Poincar\'{e} series} of the Chow ring) of $M$, which we will also denote by $ \underline{\H}_{M}(x) $.

    The Chow polynomials of uniform matroids are well-studied and have numerous combinatorial formulas. The following is only a partial
    list of results. (See Section \ref{section:combinatorics} for any undefined notation.)
    \begin{thm}
        Let $U_{r,n} $ be the uniform matroid of rank $ r \geq 1 $ on ground set $[n]$. Then
        \begin{enumerate}
            
            \item \cite[Theorem 3]{hoster}
                \[ 
                    \underline{\H}_{U_{r,n}}(x) = \sum_{\substack{R \subseteq \left\{0, \dots, r-1 \right\} \\ 0 \in R}}
                    \binom{n}{\Delta R} \cdot x^{\left| R \right|-1}.
                \]
            \item \cite[Theorem 5.1]{HRS}
                \[
                    x \cdot \underline{\H}_{U_{n-1,n}}(x) = d_{n}(x).
                \]
            \item \cite[Theorem 1.9]{FMSV}
                \[
                    \underline{\H}_{U_{r,n}}(x) = \sum_{j=0}^{r-1} \binom{n}{j} \; d_{j}(x) \cdot (1 + x + \cdots + x^{r-j-1}).
                \]
            \item \cite[Equation (13)]{FMSV} 
                \[ 
                    \underline{\H}_{U_{r,n}}(x) = \sum_{j=0}^{r-1} \sum_{m=0}^{r-j-1} (-1)^{m} \, \binom{n}{j} \binom{n-j-1}{m} \,
                    A_{j}(x) \cdot x^{r-m-j-1}.
                \]
        \end{enumerate}
    \end{thm}

    From the combinatorics of retral and weakly retral chains,
    we give combinatorial proofs for (1)--(3) and rederive (4).

    We also produce analogous formulas for the other refined Hodge--Poincar\'{e} polynomials.
    \begin{thm}\label{thm:formulasHP}
        Let $U_{r,n} $ be the uniform matroid of rank $ r \geq 1 $ on ground set $[n]$. For $ i \geq 1 $,
        \begin{enumerate}
            \item {[Theorem \ref{thm:hosterHP}]}
                \[\underline{\H}^{i}_{U_{r,n}}(x) = \sum_{\substack{ R \subseteq 
						\left\{ 0,\dots, r-1 \right\} \\ 0, r-1 \in R}}
                        \sum_{\ell=1}^{n-r-i+1} \, \binom{n}{\Delta R}\,
                        \binom{n - \lfloor R \rfloor - \ell}{r-
                        \lfloor R \rfloor -1} \binom{n-r-\ell}{i-1}
                         \cdot x^{\left| R \right| -1}.
                \]
            \item {[Theorem \ref{thm:derangementHP}]}
			    \[ 
				\underline{\H}^{i}_{U_{r,n}}(x) = \sum_{j=0}^{r-1}
                \sum_{\ell=1}^{n-r-i+1} \, \binom{n}{j} \binom{n - j- \ell}{r-j - 1} \binom{n-r-\ell}{i-1} \, d_{j}(x) \cdot x^{r-j-1}.
			    \]

            \item {[Proposition \ref{prop:recursiveRHP}]}
        		\[
                     \underline{\H}_{U_{r,n}}^{i}(x) = (-1)^{r-1}\, \binom{n-1}{i+r-1} \cdot x^{r-1} 
					 + \sum_{j=1}^{r-1}(-1)^{r+1-j} \, \binom{n}{j}\binom{n-j-1}{i+r-j-1} \,  A_{j}(x)\cdot x^{r-j}.
			    \] 

        \end{enumerate}
    \end{thm}

    \subsection{Lefschetz properties in the singular cohomology ring of a uniform matroid}
    One of the miraculous properties of the Chow ring of a matroid is that it acts like the cohomology 
    ring of a smooth projective variety even though the variety $ X_{\Sigma_{M}} $ is often non-proper
    and thus only quasi-projective. Specifically, Adiprasito, Huh, and Katz \cite{AHK} 
    proved that $ A^{\bullet}(X_{\Sigma_{M}}) $ satisfies the \emph{K\"{a}hler package} of Poincar\'{e} duality, 
    Hard Lefschetz, and the Hodge--Riemann relations. A natural question is the following.

    \begin{question}
        Does the singular cohomology ring of a matroid satisfy anything like the K\"{a}hler package?
    \end{question}

    Unfortunately, the K\"{a}hler package itself quickly fails for the singular cohomology rings of matroids. 
    By \cite[Theorem 6.6]{binder}, the Betti diagram for the singular cohomology of a simple rank $r$ matroid $M$ on 
    ground set $ [n] $ takes the form 
    \vspace*{1em}
    \begin{center}
        \begin{tabular}{c|ccccc}
            & $0$ & $1$ & $ \cdots $ & $ n-r-1$ & $ n-r $ \\ \hline
            $0$ & $1$ & $0$ & $ \cdots $ & $ 0$ & $0$ \\ 
            $1$ & * & * & $ \cdots $ & * & * \\
            $\vdots $ & $ \vdots $ & $ \vdots $& $ \ddots$& $ \vdots $ & $ \vdots $ \\
            $ r-2 $ & * & * & $ \cdots $ & *&* \\
            $ r-1$ & $1$ & * & $ \cdots $ & * & $ \mu(M) $
        \end{tabular}
    \end{center}
    \vspace*{1em}
    where asterisks indicate the locations of potentially non-zero entries. The dimension of
    the top-degree singular cohomology is $ \mu(M) $, the 
    \emph{M\"{o}bius invariant} of $M$. For loopless matroids, $ \mu(M) \geq 1 $, and often this inequality 
    is strict. This immediately implies that Poincar\'{e} duality fails in general, as the dimension the top-degree cohomology is not 1.  
    Hard Lefschetz, which implies Poincar\'{e} duality,
    thereby fails as well. Finally, as there is no degree map on the top-degree cohomology, we cannot even define the Hodge--Riemann relations.

    While there does not appear to be a way to salvage the Poincar\'{e} duality and Hodge--Riemann relations aspects of the K\"{a}hler 
    package, there are weaker Lefschetz properties which the singular cohomology ring may satisfy.
    We show in Section \ref{section:lefschetz} that the singular cohomology ring of a uniform matroid satisfies 
    what we call the \emph{quasi-projective strong Lefschetz property}.

    \begin{thm}
        Let $ \ell \in \Gr_{2}^{W}H^{2}(X_{\Sigma_{U_{r,n}}}) $ be generic. Then the map
        \[ 
            \times \ell^{d} \colon \Gr_{j}^{W}H^{k}(X_{\Sigma_{U_{r,n}}}) \to
            \Gr_{j+2d}^{W}H^{k+2d}(X_{\Sigma_{U_{r,n}}})
        \]
        is injective for $ 0 \leq d \leq r + j-2k-1 $ and surjective for $ d > r + j-2k-1 $.
    \end{thm}
    Our proof uses the construction of our basis for the singular cohomology ring to deduce this
    Lefschetz property from the Hard Lefschetz property of the permutohedral variety.

    This strong Lefschetz property also has an interesting combinatorial implication: Each refined Hodge--Poincar\'{e}
    polynomial $ \underline{\H}_{U_{r,n}}^{i}(x) $ has unimodal coefficients. In fact, using the formulas from
    Theorem \ref{thm:formulasHP}, we have verified that the polynomials $\underline{\H}^{i}_{U_{r,n}}(x) $ have ultra-log-concave coefficients
    for $ n \leq 230 $. This suggests these polynomials may be real-rooted.

    \subsection*{Acknowledgements}
    This work was supported by NSF RTG grants DMS-2231492 and DMS-2231565. We thank Christin Bibby and Eric Katz for helpful conversations.

    \section{Background and notation}\label{section:notation}
    We assume that the reader is familiar with matroids and toric varieties, and we mostly set notation here. We refer to 
    \cite{oxley} and \cite{CLS}, respectively, for further details.

    \subsection{Matroids and Bergman fans}
     Let $M$ be a loopless matroid of rank $ r $ on ground set $ [n] $.
     We write $ \mathscr{L}(M) $ for the 
     \emph{lattice of flats} of $M$. A \emph{proper flat} is flat of rank at most $ r-1 $,  
     and a \emph{hyperplane} is flat of rank exactly $ r-1 $. We write $ \mathscr{H}(M) $ for the 
     set of hyperplanes of $ M $. For a flat $ F \in \mathscr{L}(M) $, we write the \emph{restriction
     of $ M $ to $ F $} as $ M^{F} $. 

     Throughout this paper we will mostly be restricting ourselves to uniform matroids.
     For integers $ 1 \leq r \leq n $, we write $U_{r,n} $ for the uniform matroid of rank $r$ on 
     ground set $[n] $. 

     Associated to $ M $ is a rational fan in $ \mathbb{R}^{n}/ \mathbb{R} \cdot (1,\dots, 1) $
     called the \emph{Bergman fan}, $\Sigma_{U_{r,n}} $. 
     To construct it, let $ N $ be the lattice $ \mathbb{Z}^{n}/ \mathbb{Z} \cdot (1,\dots,1) $, and let
     $ e_{i} $ be the image of the $i$th basis vector of $ \mathbb{Z}^{n} $ in $N$. For a proper, non-empty flat
     $ F \in \mathscr{L}(M) $, define $ e_{F} = \sum_{i \in F} e_{i} $. The rays of 
     $ \Sigma_{M} $ are given by $ \rho_{F} = \pos(e_{F}) $ for each proper, non-empty flat $F$. 
     The higher-dimensional cones in $ \Sigma_{M} $ are given by chains of flats in $ \mathscr{L}(M) $:
     $ \pos(\rho_{F_{1}}, \dots, \rho_{F_{s}}) \in \Sigma_{M} $ if and only if (up to reordering)
     $ F_{1} < \cdots < F_{s} $. Figure \ref{fig:BergmanFan} depicts $ \Sigma_{U_{3,3}} $.

          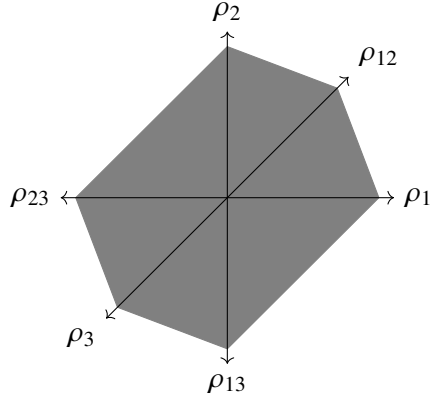
\begin{figure}
        \begin{tikzpicture}
            \filldraw[gray]  (2,0) -- (1.45, 1.45) -- (0,2) -- (-2,0) -- (-1.45,-1.45) -- (0,-2) -- cycle;
            \draw[->] (0,0) -- (2.2,0) node[anchor = west] {$\rho_{1}$};
            \draw[->] (0,0) -- (1.6,1.6) node[anchor = south west] {$\rho_{12}$};
            \draw[->] (0,0) -- (0,2.2) node[anchor = south] {$\rho_{2}$};
            \draw[->] (0,0) -- (-2.2,0) node[anchor = east] {$\rho_{23}$};
            \draw[->] (0,0) -- (-1.6,-1.6) node[anchor = north east] {$\rho_{3}$};
            \draw[->] (0,0) -- (0,-2.2) node[anchor = north] {$\rho_{13}$};
        \end{tikzpicture}
        \caption{The Bergman fan of $U_{3,3}$.}\label{fig:BergmanFan}
     \end{figure}

    \subsection{Singular cohomology of smooth toric varieties}
     Let $N$ be a lattice and $ N_{\mathbb{R}} = N \otimes \mathbb{R} $. Given a unimodular fan
     $ \Sigma \subseteq N_{\mathbb{R}} $, we write $ X_{\Sigma} $ for the associated smooth toric 
     variety defined over $ \mathbb{C} $. We denote the character lattice of $ X_{\Sigma} $ by  $N^{\vee} $
     and define $ N^{\vee}_{\mathbb{Q}} = N^{\vee} \otimes \mathbb{Q} $. 

    The singular cohomology of $ X_{\Sigma} $ (with rational coefficients)
    has a description in terms of the Stanley--Reisner ring of $ \Sigma $ and Tor algebras.

    \begin{definition}
        The \emph{Stanley--Reisner ring} of $ \Sigma $ is the quotient ring
        \[ 
            \mathbb{Q}[\Sigma] = \frac{\mathbb{Q}[x_{\rho} : \textrm{$ \rho $ a ray in $ \Sigma$}]}
                {(x_{\rho_{1}} \cdots x_{\rho_{s}} : \pos(\rho_{1}, \dots, \rho_{s}) \notin \Sigma).}
        \]
    \end{definition}
    The rational embedding $ \Sigma \subseteq N_{\mathbb{R}} $ endows $ \mathbb{Q}[\Sigma] $ with the 
    structure of a $ \Sym(N_{\mathbb{Q}}^{\vee}) $-algebra via the map
    \[ 
        N^{\vee} \longrightarrow \mathbb{Q}[\Sigma]\;\;\;\;\;\;\; m \longmapsto \sum_{\textrm{$ \rho $ ray}} 
        \langle m, u_{\rho} \rangle \cdot x_{\rho},
    \]
    where $ \langle -,- \rangle $ is the pairing $ N^{\vee} \otimes N $ and $ u_{\rho} \in N  $ is the primitive 
    generator of the ray $ \rho $.

    \begin{thm}[\cite{franzRing}]
        If $ X_{\Sigma} $ is a smooth toric variety, there is a natural ring isomorphism
        \[ 
            H^{\bullet}(X_{\Sigma}, \mathbb{Q}) \cong \Tor_{\bullet}^{\Sym(N_{\mathbb{Q}^{\vee}})}
            \left( \mathbb{Q}[\Sigma], \mathbb{Q}\right)_{\bullet},
    \]
        where $ \mathbb{Q} $ is the residue field of the maximal homogeneous ideal.
    \end{thm}

    The singular cohomology has the additional structure of the \emph{weight filtration}
    \[ 
        0 = W_{i-1} H^{i}(X_{\Sigma}) \subseteq W_{i}H^{i}(X_{\Sigma}) \subseteq \cdots \subseteq 
        W_{2i-1}H^{i}(X_{\Sigma}) \subseteq W_{2i}H^{i}(X_{\Sigma}) = H^{i}(X_{\Sigma}).
    \]
    By \cite{totaro}, the weight filtration splits, and by \cite{weberWeights}, the associated gradeds correspond to 
    the bigrading of the Tor algebra
    \[ 
        \Gr_{2j}^{W}H^{2j-i}(X_{\Sigma}) \cong \Tor_{i}^{\Sym(N_{\mathbb{Q}}^{\vee})}(\mathbb{Q}[\Sigma], \mathbb{Q})_{j}.
    \]
    We call the dimensions of the associated gradeds \emph{Hodge numbers}. 

    One obtains a combinatorial description of the singular cohomology by computing the Tor algebra from the Koszul resolution
    of $ \mathbb{Q} $. Let $ \mathbb{Q}[\Sigma]_{i} $ denote the space of degree $i$ homogeneous forms.
    \begin{definition}
        The \emph{Koszul complex} of $ \mathbb{Q}[\Sigma] $ is the bigraded differential algebra
        \[ 
            K_{i}(\mathbb{Q}[\Sigma])_{j} = \mathbb{Q}[\Sigma]_{j-i} \otimes_{\mathbb{Q}} \bigwedge^{i} N_{\mathbb{Q}}^{\vee}
        \]
        with differential $ d $ defined by the Leibniz rule, $ d(x_{\rho}) = 0 $, and $ d(m) = \sum_{\textrm{$\rho $ ray}}
        \langle m, u_{\rho} \rangle \cdot x_{\rho} $. The algebra structure is induced from the algebra structure on
        $ \mathbb{Q}[\Sigma] $ and $ \bigwedge^{\bullet} N_{\mathbb{Q}}^{\vee} $.

        Taking homology of the Koszul complex yields \emph{Koszul homology} $H_{\bullet}(K(\mathbb{Q}[\Sigma]))_{\bullet} $, and
        we have isomorphisms of graded rings
        \[ 
            \Gr^{W}_{2j}H^{2j-i}(X_{\Sigma}) \cong \Tor_{i}^{\Sym(N_{\mathbb{Q}}^{\vee})}(\mathbb{Q}[\Sigma], \mathbb{Q})_{j} 
            \cong H_{i}(K(\mathbb{Q}[\Sigma]))_{j}.
        \]
    \end{definition}

    \subsection{Singular cohomology of matroids}

    The main object of our paper is the following.

    \begin{definition}
        For a loopless matroid $M$, the \emph{singular cohomology ring} of $M$ is the singular cohomology ring
        of the toric variety $ X_{\Sigma_{M}} $.
    \end{definition}

        We use Koszul homology to work with these cohomology rings, and we adopt some notation for our purposes.

        The indeterminates of $ \mathbb{Q}[\Sigma_{M}] $ are in bijection with the proper, non-empty flats in 
        $ \mathscr{L}(M) $. For a flat $ F $ we write $ x_{F} $ for the corresponding indeterminate.
        The square-free monomials in $ \mathbb{Q}[\Sigma_{M}] $ are in bijection with chains of proper, non-empty flats in
        $ \mathscr{L}(M) $. For such a chain $ \mathscr{F} $, we define
        $ x_{\mathscr{F}} = \prod_{F \in \mathscr{F}} x_{F} $. Often we will want to consider chains containing the empty flat,
        so by convention we define $ x_{\emptyset} = 1 $ and extend the definition of $ x_{\mathscr{F}} $ to chains of proper flats.

        For $ \Sigma_{M} $, $ N_{\mathbb{Q}}^{\vee} $ is spanned by elements $ \ell_{i} - \ell_{j} $ for 
        $ i,j \in [n] $, where the $\ell_{i} $ form  a dual basis to the canonical vectors $ e_{i} \in \mathbb{Z}^{n} $.

        We now give some explicit examples of computations in the Koszul complex.
        \begin{example}\label{ex:U1n}
         For $ U_{1,n} $, $ \mathbb{Q}[\Sigma_{U_{1,n}}] \cong \mathbb{Q} $. Therefore
         $ K_{\bullet}(\mathbb{Q}[\Sigma_{U_{1,n}}]) \cong \bigwedge^{\bullet} N_{\mathbb{Q}}^{\vee} $ with trivial differential,
         and $H_{\bullet}(\mathbb{Q}[\Sigma_{U_{1,n}}]) \cong \bigwedge^{\bullet} N_{\mathbb{Q}}^{\vee} $.
        \end{example}

        \begin{example}
            For $ U_{2,3} $, $ \mathbb{Q}[\Sigma_{U_{2,3}}] \cong \frac{\mathbb{Q}[x_{1}, x_{2}, x_{3}]}{x_{1} x_{2}, x_{1} x_{3}, x_{2}x_{3}} $.
            By \cite[Lemma A.5]{binder}, we can restrict to the quasi-isomorphic \emph{square-free Koszul complex} with generators
            \[ 
                1, x_{1}, x_{2}, x_{3}, (\ell_{1}- \ell_{2}), (\ell_{2}-\ell_{3}), (\ell_{1} -\ell_{2}) \wedge (\ell_{2} -\ell_{3}), 
                x_{1} \otimes (\ell_{2} - \ell_{3}), x_{2} \otimes (\ell_{1} - \ell_{3}), x_{3} \otimes (\ell_{1} - \ell_{2}).
            \]
            The only non-zero differentials are
            $ d(\ell_{1} - \ell_{2}) = x_{1} - x_{2} $, $ d(\ell_{2} - \ell_{3}) = x_{2} - x_{3}$, and 
            $d((\ell_{1}- \ell_{2}) \wedge (\ell_{2} - \ell_{3})) = x_{1} \otimes (\ell_{2} - \ell_{3}) 
                - x_{2} \otimes (\ell_{1} - \ell_{3}) + x_{3} \otimes (\ell_{1} - \ell_{2}) $.

            Therefore, $ H_{\bullet}(\mathbb{Q}[\Sigma_{U_{2,3}}]) $ has a basis
            $ 1 $, $ x_{3} $, $x_{2} \otimes (\ell_{1} - \ell_{3})$, $ x_{3} \otimes (\ell_{1} - \ell_{2})$.
        \end{example}

        \begin{example}
            Take the element $ \alpha = x_{2} x_{235} \otimes (\ell_{3}-\ell_{5}) \wedge (\ell_{1} - \ell_{6}) \in 
            K_{2}(\mathbb{Q}[\Sigma_{U_{5,6}}])_{4} $. We compute its differential to be
            \[ 
                d(\alpha) = x_{2} x_{23} x_{235} \otimes (\ell_{1} - \ell_{6}) - x_{2} x_{25} x_{235} \otimes (\ell_{1}- \ell_{6})
                - x_{2} x_{235} x_{1235} \otimes (\ell_{3} - \ell_{5}) + x_{2} x_{235} x_{2356} \otimes (\ell_{3} - \ell_{5}).
            \]
            On the other hand, the chain $ \beta = x_{3} x_{1234} \otimes (\ell_{5} - \ell_{6}) \in K_{1}(\mathbb{Q}[\Sigma_{U_{5,6}}])_{3} $
            is a cycle.
        \end{example}

\section{A basis for the singular cohomology ring of a uniform matroid}\label{section:basis}
    In this section we construct an explicit, combinatorial basis for the singular cohomology ring of a uniform 
    matroid in terms of Koszul homology $H_{\bullet}(K(\mathbb{Q}[\Sigma_{U_{r,n}}]))$.

\subsection{Retral and weakly retral chains}
    We describe our basis for the singular cohomology ring in terms of \emph{retral} and \emph{weakly retral} chains of 
    flats in $ \mathscr{L}(U_{r,n}) $ which we define here. Roughly speaking, the retral property is
    a maximality condition which is enforced on each covering relation appearing in the chain. The weakly  retral
    condition is the same, only it is \emph{not} enforced for hyperplanes covering another flat.

	\begin{definition}
        For proper flats $ F, F' $ in $ \mathscr{L}(U_{r,n}) $, we say that $F'$ \emph{covers} $ F $ 
        if $ F < F' $ and there is no $ F''  \in \mathscr{L}(U_{r,n}) $ with 
		$ F < F'' < F' $. Moreover, $ F' $ is the \emph{maximal cover} of $ F $ if $F'$ is
		the lexicographically largest flat which covers $ F$. As $ F' \setminus F $ is a singleton whenever
        a proper flat $ F' $ covers $F$, this last condition is equivalent to $ \max ([n] \setminus F) = F' \setminus F $.
	\end{definition}

    For notation, we will always assume that a chain of proper flats in $ \mathscr{L}(U_{r,n}) $ 
    contains the empty flat. We also write $ \max(\mathscr{F}) $ for the maximal element in $ \mathscr{F} $.
	
	\begin{definition}\label{def:retralChains}
        Let $ \mathscr{F} = \left\{ \emptyset = F_{0} < F_{1} < \cdots < F_{s} \right\} $ be a chain of proper flats
        in $ \mathscr{L}(U_{r,n}) $, and extend notation so that $ F_{s+1} =[n] $. We say that
        $ \mathscr{F} $ is \emph{retral} in $ \mathscr{L}(U_{r,n}) $ if, for all $ 1 \leq i \leq s $,
        $ F_{i} $ covers $ F_{i-1} $ if and only if $F_{i}$ is the maximal cover of $ F_{i-1} $ in
		$ \mathscr{L}(U_{r,n}^{F_{i+1}}) $.
	\end{definition}

	\begin{example}
      Below is the set of all retral chains in $ \mathscr{L}(U_{3,4}) $.
        \begin{gather*}
            \emptyset, \\ \emptyset < 4, \; \emptyset < 12, \; \emptyset < 13, \;  \emptyset < 14, \; 
            \emptyset < 23, \; \emptyset < 24, \; \emptyset < 34, \\
            \emptyset < 4 < 34 
        \end{gather*}
      We note that all chains $ \emptyset < F $ for $ \rk(F) = 2 $ are retral as there are no covering relations in
      the chain; therefore, the maximal cover condition is vacuous.
    \end{example}

    \begin{example}
        The chain $ \emptyset < 3 < 123 < 1235 $ is a retral chain in $ \mathscr{L}(U_{5,5}) $. 
        Note, however, that $3$ is not the maximal cover of $ \emptyset $ in $ \mathscr{L}(U_{5,5}) $.
        The maximal cover condition only requires that $ 3 $ be the maximal cover of $ \emptyset $ 
        in $ \mathscr{L}(U_{5,5}^{123}) $.
    \end{example}

	For a chain of proper flats  containing a hyperplane, 
	we will remove the maximal cover condition on the hyperplane to define \emph{weakly retral chains}.

	\begin{definition}\label{def:weaklyRetralChains}
        Let $ \mathscr{F} = \left\{ \emptyset =F_{0} < F_{1} < \cdots < F_{s} \right\} $ be a chain of proper flats  
        in $ \mathscr{L}(U_{r,n}) $ such that $ \max(\mathscr{F}) = F_{s} $ is a hyperplane. 
        We say that $ \mathscr{F} $ is \emph{weakly retral} in $ \mathscr{L}(U_{r,n}) $ if,
        for all $ 1 \leq i < s $, $ F_{i} $ covers $ F_{i-1} $ if and only if $F_{i}$ is the maximal
		cover of $ F_{i-1} $ in  $ \mathscr{L}(U_{r,n}^{F_{i+1}}) $.
	\end{definition}

	\begin{example}
	    The weakly retral chains in $ \mathscr{L}(U_{3,4}) $ are 
        \begin{gather*}
            \emptyset < 12, \; \emptyset < 13, \; \emptyset < 14, \; 
            \emptyset < 23, \; \emptyset < 24, \; \emptyset < 34, \\
            \emptyset < 2 < 12, \; \emptyset < 3 <  13, \; \emptyset < 4 <  14, \; 
            \emptyset < 3 < 23, \; \emptyset <  4 < 24, \; \emptyset <4 <  34.
        \end{gather*}
        Note that, even though $ \emptyset < 4 $ is a retral chain, it is not weakly retral as
        it does not contain  a hyperplane.
	\end{example}

    For a proper, non-empty flat $ F $, we define retral and weakly retral chains of 
    flats in $ \mathscr{L}(U_{r,n}^{F}) $.

    \begin{definition}
        A chain of proper flats $ \mathscr{F} $ in $ \mathscr{L}(U_{r,n}^{F}) $ 
        is \emph{retral} (\emph{weakly retral}) if and only if it is  the image of a retral (weakly retral) chain
        of $ \mathscr{L}(U_{\left| F \right|, \left|F \right|}) $ under the unique lexicographic-order preserving
        isomorphism $ U_{\left| F \right|, \left| F \right|} \to U_{r,n}^{F} $.
    \end{definition}

    We need one final definition before describing our basis for the singular cohomology ring.

	\begin{definition}\label{def:admissible}
            Let $ \mathscr{F} $ be a weakly retral chain in $ \mathscr{L}(U_{r,n}) $
			containing the hyperplane $ \max(\mathscr{F}) $.
			A pure wedge $ \xi = (\ell_{u_{1}} - \ell_{v_{1}}) \wedge \cdots \wedge
				(\ell_{u_{i}} - \ell_{v_{i}}) \in \bigwedge^{i >0 } N^{\vee}_{\mathbb{Q}} $
			is \emph{admissible} to $ \mathscr{F} $ if:
			\begin{enumerate}
                \item Each $ u_{m} $ and $ v_{m} $ is in $ [n] \setminus \max(\mathscr{F})  $.
				\item $ u_{1} < \dots < u_{i} $, and for all $m$, $v_{m} $ is the successor of
                        $ u_{m} $ in the  ordered set $ [n] \setminus \max(\mathscr{F}) $.
                \item If $ \max(\mathscr{F}) $ covers $ F \in \mathscr{F} $, then
                        $ u_{1} $ is less than the single element of $ \max(\mathscr{F}) \setminus F $.
			\end{enumerate}
	\end{definition}

    We note that condition (1) ensures that $ x_{\mathscr{F}} \otimes \xi $ is a cycle in the
    Koszul complex, and condition (2) ensures the collection of admissible wedges to $ \mathscr{F} $ are linearly independent.
    Only condition (3) has non-trivial content for describing our basis.

    \begin{example}
        For the weakly retral chain $ \mathscr{F} = \left\{ \emptyset < 4 < 24 \right\} $ in $ \mathscr{L}(U_{3,4}) $, the only admissible wedge is
        $ (\ell_{1} - \ell_{3}) $.
        On the other hand, the weakly retral chain $ \mathscr{F} = \left\{ \emptyset < 4 < 14 \right\} $ has no admissible wedges.
    \end{example}

    \begin{example}
        If $ \mathscr{F} $ is a weakly retral chain in $ \mathscr{L}(U_{n,n}) $, then there are no admissible wedges.
        This is because $ [n] \setminus \max(\mathscr{F}) $ is a singleton, so conditions (1) and (2) cannot be 
        simultaneously satisfied.
    \end{example}

	\subsection{A basis of the singular cohomology ring}

        We now describe and construct our basis for the singular cohomology ring of a uniform matroid.

	\begin{thm}\label{thm:basis}
		The set
		\[ 
			\left\{ x_{\mathscr{F}} : \textrm{$\mathscr{F}$ is retral} \right\} \cup
					\left\{ x_{\mathscr{F}} \otimes \xi : \textrm{$ \mathscr{F} $ is weakly
					retral and $ \xi $ is admissible to $ \mathscr{F} $} \right\}
		\]
		is a basis of $ H_{\bullet}(K(\mathbb{Q}[\Sigma_{U_{r,n}}]))
		\cong H^{\bullet}(X_{\Sigma_{U_{r,n}}}) $ 
		as a $ \mathbb{Q} $-vector space.
	\end{thm}

	 We will prove Theorem \ref{thm:basis} through the following five lemmas, inducting on $ r $. We
	 start with the base case.

	 \begin{lem}\label{lem:baseCase}
			Theorem \ref{thm:basis} holds in rank $ 1 $.	
	 \end{lem}

	 \begin{proof}
		From Example \ref{ex:U1n}, 
		$ H_{\bullet}(K(\mathbb{Q}[\Sigma_{U_{1,n}}])) \cong
		\bigwedge^{\bullet} N_{\mathbb{Q}}^{\vee} $.
		The chain $\emptyset $ is the only weakly retral chain (it is also retral), and the 
		wedges which are admissible to $ \emptyset $ form a basis of 
		$ \bigwedge^{>0} N_{\mathbb{Q}}^{\vee} $. 
	 \end{proof}

	 For the inductive step of the proof, we use the
		\emph{truncation exact sequence} 
		\begin{equation}\label{eq:truncationLES}
            \cdots \to \bigoplus_{H \in \mathscr{H}(U_{r,n})} 
			H_{i}(K(\mathbb{Q}[\overline{\Star}(\rho_{H})] ) )_{j}
			\to H_{i} ( K ( \mathbb{Q}[\Sigma_{U_{r,n}}] ) )_{j+1} \to 
			H_{i}( K ( \mathbb{Q}[\Sigma_{U_{r-1,n}}] ) )_{j+1} \to \cdots
		\end{equation}
		where the subfan $ \overline{\Star}(\rho_{H}) = \left\{ \sigma \in \Sigma_{U_{r,n}} : 
		\pos(\sigma, \rho_{H} ) \in \Sigma_{U_{r,n}} \right\} $ is the \emph{closed star} of
		$ \rho_{H} $. 	The exact sequence is induced by applying Koszul homology to the exact sequence of 
	  $ \Sym(N_{\mathbb{Q}}^{\vee}) $-algebras
	  \[ 
          0 \to \bigoplus_{H \in \mathscr{H}(U_{r,n})} \mathbb{Q}[\overline{\Star}(\rho_{H})]
			\xrightarrow{\oplus \cdot x_{\rho_{H}}} \mathbb{Q}[\Sigma_{U_{r,n}}] \to
			\mathbb{Q}[\Sigma_{U_{r-1,n}}] \to 0.
	 	\]
	 	We will write the  connecting homomorphism as
	 	\[ 
            \delta_{i} \colon H_{i}( K(\mathbb{Q}[\Sigma_{U_{r-1,n}}] ))_{j}  \to
			\bigoplus_{H} H_{i-1}(K(\mathbb{Q}[\overline{\Star}(\rho_{H})]))_{j-1}.
	 	\]

	 The inductive hypothesis provides
	 bases for $ H_{i}( K(\mathbb{Q}[\Sigma_{U_{r-1,n}}]))_{j} $ and each 
	 $ H_{i}(K(\mathbb{Q}[\overline{\Star}(\rho_{H})]))_{j} $ 
	 which we describe in Lemma \ref{lem:hyperplaneBasis}. 
	 In Lemmas \ref{lem:initialBasis} and \ref{lem:initInjective}, we use these bases to 
     compute a Gr\"{o}bner basis for the image of the connecting homomorphism $ \delta_{i} $. We describe
     the standard monomials for $ \coker \delta_{i} $ in Lemma \ref{lem:standardMonomials},
	 and these produce the basis for Theorem \ref{thm:basis}.

	 \begin{definition}
         For a hyperplane $ H \in \mathscr{H}(U_{r,n}) $, define $P_{H} \subseteq \bigwedge^{>0} N_{\mathbb{Q}}^{\vee} $ to be the set of pure 
			wedges $ \xi = (\ell_{u_{1}} - \ell_{v_{1}}) \wedge \cdots \wedge (\ell_{u_{i}} -
					 \ell_{v_{i}}) $ such that
			\begin{enumerate}
				\item Each $ u_{m} $ and $ v_{m} $ is in $ [n] \setminus H $.
				\item $ u_{1} < \cdots <  u_{i}$, and for all $ m $, $v_{m} $ is the 
							successor of $ u_{m} $ in the ordered set $[n] \setminus H $.
			\end{enumerate}
		\end{definition}

	 \begin{lem}\label{lem:hyperplaneBasis}
			Assume that Theorem \ref{thm:basis} holds in rank $ r- 1 $. Then
			\[ 
					\left\{ x_{\mathscr{F}} \otimes \xi : \textrm{$ \mathscr{F} $ is retral in 
					 $ \mathscr{L}(U_{r,n}^{H}) $ and $ \xi \in P_{H} $} \right\} 
			\]
            is a basis for $H_{\bullet}(K(\mathbb{Q}[\overline{\Star}(\rho_{H})])) $, where $ H \in \mathscr{H}(U_{r,n})$.
		\end{lem}

		\begin{proof}
			By \cite[Lemma A.3]{binder},  $
			H_{\bullet}(K(\mathbb{Q}[\Star(\rho_{H})])) \cong
			H_{\bullet}(K(\mathbb{Q}[\overline{\Star}(\rho_{H})])) $, where 
			$ \Star(\rho_{H}) $ is the \emph{star} of $ \rho_{H} $ in $ \Sigma_{U_{r,n}} $
			(see \cite[Definition (3.2.8)]{CLS}). 
			This isomorphism is induced by the natural inclusion of complexes
			$ K_{\bullet}(\mathbb{Q}[\Star(\rho_{H})]) \hookrightarrow
			K_{\bullet}(\mathbb{Q}[\overline{\Star}(\rho_{H})]) $. In our case,
			$ \Star(\rho_{H}) \cong  \Sigma_{U_{r-1,r-1}} \times 0 $ in the product vector space $ N_{H,\mathbb{R}}
            \times N_{[n] \setminus H, \mathbb{R}} $ where $N_{H, \mathbb{R}} = \spann(e_{i} :
            i \in H)/e_{H} $ and $ N_{[n] \setminus H, \mathbb{R}} = {\spann(e_{i}  : i \in [n] \setminus H)}/e_{[n] \setminus H} $.
			Therefore, $ K_{\bullet}(\mathbb{Q}[\Star(\rho_{H})]) $ is isomorphic to the tensor
			complex $ K_{\bullet}(\mathbb{Q}[\Sigma_{U_{r-1,r-1}}]) \otimes_{\mathbb{Q}}
            \bigwedge^{\bullet} N_{[n] \setminus H, \mathbb{Q}}^{\vee} $, with trivial differential
            on $ \bigwedge^{\bullet} 
			N_{[n] \setminus H, \mathbb{Q}}^{\vee} $ . As $ P_{H} $ forms 
			a basis for $ N_{[n] \setminus H, \mathbb{Q}}^{\vee} $, the lemma follows from the
			K\"{u}nneth formula and the basis for $H_{\bullet}(K(\mathbb{Q}[\Sigma_{U_{r-1,r-1}}])) $
			described in Theorem \ref{thm:basis}.
		\end{proof}

        For each $H \in \mathscr{H}(U_{r,n}) $, choose an arbitrary ordering of the basis elements from
		Lemma \ref{lem:hyperplaneBasis}. Then extend this to an ordering of the basis elements of
        $ \bigoplus_{H} H_{\bullet}(K(\mathbb{Q}[\overline{\Star}(\rho_{H})])) $
		such that a basis element of $H_{\bullet}(K(\mathbb{Q}[\overline{\Star}(\rho_{H_{1}})])) $
		precedes a basis element of $H_{\bullet}(K(\mathbb{Q}[\overline{\Star}(\rho_{H_{2}})])) $
		if and only if $ H_{1} <_{\lex} H_{2} $.

        We now study the connecting homomorphisms $ \delta_{i} $ through Gr\"{o}bner theory.
		For a basis element $ \alpha  $ of $ H_{i}(K(\mathbb{Q}[\Sigma_{U_{r-1,n}}])) $,
		we are interested in the first basis element of 
        $ \bigoplus_{H} H_{i-1}(K(\mathbb{Q}[\overline{\Star}(\rho_{H})])) $
        that appears in $ \delta_{i}(\alpha) $.

		\begin{definition}
			For a basis element $\displaystyle \alpha = x_{\mathscr{F}} \otimes \bigwedge_{m =1}^{i}
				(\ell_{u_{m}} - \ell_{v_{m}}) \in H_{i}(K(\mathbb{Q}[\Sigma_{U_{r-1,n}}]))$,
			define \[ \init(\alpha) =  x_{\mathscr{F}} \otimes
			\bigwedge_{m = 2}^{i} (\ell_{u_{m}} - \ell_{v_{m}})
                \in H_{i-1}(K(\mathbb{Q}[\overline{\Star}(\rho_{H})])),\]
                where $  H = \max(\mathscr{F})\cup u_{1} $ is a hyperplane of $ U_{r,n} $.
		\end{definition}

		\begin{lem}\label{lem:initialBasis}
			Assume that Theorem \ref{thm:basis} holds in rank $r-1 $.
			Then for each basis element $ {\alpha \in H_{i>0}(K(\mathbb{Q}[\Sigma_{U_{r-1,n}}]))} $, 
            $ \init(\alpha) $ is the first basis element appearing in $ \delta_{i}(\alpha) $.
		\end{lem}

		\begin{proof}
			We first show that $ \init(\alpha) $ is a basis element. Write $ \alpha =
			x_{\mathscr{F}} \otimes \bigwedge_{m=1}^{i}(\ell_{u_{m}} - \ell_{v_{m}}) $ 
			and $ H = \max(\mathscr{F}) \cup u_{1} $ as above. By assumption, $ \mathscr{F} $ 
			is a weakly retral chain in $ \mathscr{L}(U_{r-1,n}) $, and we must show that it is retral
			in $ \mathscr{L}(U_{r,n}^{H}) $. The only thing to check is the maximal cover condition for
			$ \max(\mathscr{F}) $. If $ \max(\mathscr{F}) $ covers $ F \in \mathscr{F} $,
			then $ u_{1} < \max(\mathscr{F}) \setminus F $ by the definition of admissible wedges.
			The two covers of $ F $ in $ \mathscr{L}(U_{r,n}^{H}) $ are therefore $ F \cup u_{1} $ and 
			$ \max(\mathscr{F}) $, and $ \max(\mathscr{F}) $ is the maximal one.
			Finally, one immediately verifies that $ \bigwedge_{m=2}^{i} (\ell_{u_{m}} - \ell_{v_{m}}) 
			\in P_{H} $.

            We now show that $\init(\alpha) $ is the first basis element appearing in $ \delta_{i}(\alpha)$.
			For a hyperplane $H_{1}$, denote the inclusion
			$ \iota_{H_{1}}\colon H_{\bullet}(K(\mathbb{Q}[\overline{\Star}(\rho_{H_{1}})]))
				\hookrightarrow \bigoplus_{H} H_{\bullet}(K(\mathbb{Q}[\overline{\Star}(\rho_{H})])) $.
            Then $ \delta_{i} $ is given by
			\[
                \delta_{i}(\alpha) = \sum_{m=1}^{i}(-1)^{m+1} \left[\iota_{\max(\mathscr{F}) \cup u_{m}} 
				\left( x_{\mathscr{F}} \otimes \bigwedge_{p\neq m} (\ell_{u_{p}} - \ell_{v_{p}}) \right) 
                -\iota_{\max(\mathscr{F}) \cup v_{m}} 
                \left( x_{\mathscr{F}} \otimes \bigwedge_{p\neq m} (\ell_{u_{p}} - \ell_{v_{p}}) \right)\right].
			\]
			By the definition of admissible wedges, $ \init(\alpha) $ is the only term in 
            $ \delta_{i}(\alpha) $ involving $ \max(\mathscr{F}) \cup u_{1}  $. 
            This is the lexicographically minimal hyperplane involved in $ \delta_{i}(\alpha) $.		
		\end{proof}

		\begin{lem}\label{lem:initInjective}
			Assume that Theorem \ref{thm:basis} holds in rank $r-1 $. Then
			$ \alpha \mapsto \init(\alpha) $ defines an injective map from basis elements 
			of $ H_{i>0}(K(\mathbb{Q}[U_{r-1,n}])) $ to basis elements of
			$ \bigoplus_{H} H_{i-1}(K(\mathbb{Q}[\overline{\Star}(\rho_{H})])) $.
		\end{lem}

		\begin{proof}
			If $ \alpha $ is a basis element of $ H_{i>0}(K(\mathbb{Q}[\Sigma_{U_{r-1,n}}])) $
			with $ \init(\alpha) = x_{\mathscr{F}} \otimes \bigwedge_{m=2}^{i} (\ell_{u_{m}} -
			\ell_{v_{m}}) $ in $ H_{i-1}(K(\mathbb{Q}[\overline{\Star}(\rho_{H})])) $, let
			$ u_{1} = H \setminus \max(\mathscr{F}) $ and
			$ v_{1} $ the successor of $ u_{1} $ in $ [n] \setminus \max(\mathscr{F}) $.
			Then we find that $  \alpha = x_{\mathscr{F}} \otimes \bigwedge_{m=1}^{i} (\ell_{u_{m}} -
						\ell_{v_{m}}) $.
		\end{proof}

		A consequence of Lemmas \ref{lem:initialBasis} and \ref{lem:initInjective} is
		that, for $ i > 0 $,
		\[ 
            \left\{ \delta_{i}(\alpha) : \textrm{$\alpha $ is a basis element of 
            $H_{i}(K(\mathbb{Q}[\Sigma_{U_{r-1,n}}]))$} \right\} 
		\]
        is a Gr\"{o}bner basis for $ \im \delta_{i} $. 

		\begin{lem}\label{lem:standardMonomials}
			Assume that Theorem \ref{thm:basis} holds in rank $ r -1 $.
			With respect to the above Gr\"{o}bner basis, the standard monomials for
            $ \coker \delta_{i} $ are the following basis elements of
			$ H_{i-1}(K(\mathbb{Q}[\overline{\Star}(\rho_{H})])) $.

            \noindent For $ i > 1 $,
            \[\bigcup_{H \in \mathscr{H}(U_{r,n})} \left\{ x_{\mathscr{F}} \otimes \bigwedge_{m=1}^{i-1}
										(\ell_{u_{m}} - \ell_{v_{m}}) 
										\in H_{i-1}(K(\mathbb{Q}[\overline{\Star}(\rho_{H})])) 
										\middle\vert 
										\begin{array}{l}\textrm{1. $H$ does not cover $ \max(\mathscr{F}) $}, or \\
                                    \textrm{2. $ u_{1} < H \setminus \max(\mathscr{F})$} \end{array} \right\}.\] 
            For $i = 1 $,
            \[ \bigcup_{H \in \mathscr{H}(U_{r,n})} \left\{ x_{\mathscr{F}}
							\in H_{0}(K(\mathbb{Q}[\overline{\Star}(\rho_{H})])) \middle\vert
									\begin{array}{l}\textrm{1. $H$ does not cover $ \max(\mathscr{F}) $, 
													or} \\ \textrm{2. $ H $ is the maximal cover of $ \max(\mathscr{F}) $ 
													in $ \mathscr{L}(U_{r,n}) $}
									\end{array}
                            \right\}.\]
		\end{lem}

		\begin{proof}
			By definition, the standard monomials are the basis elements of $ \bigoplus_{H} 
			H_{\bullet}(K(\mathbb{Q}[\overline{\Star}(\rho_{H})])) $ which are not
			$ \init(\alpha) $ for any basis element $ \alpha \in H_{>0}(K(\mathbb{Q}[\Sigma_{U_{r-1,n}}])) $.
            We will prove that the basis elements of the form $ \init(\alpha) $ are precisely those not listed in
            the statement.

			For notation, write $ \alpha = x_{\mathscr{F}} \otimes \bigwedge_{m=1}^{i} (\ell_{u_{m}} -\ell_{v_{m}}) $
			for a basis element of $ H_{i}(K(\mathbb{Q}[\Sigma_{U_{r-1,n}}])) $, and
			let $ H = \max(\mathscr{F}) \cup u_{1} $. 

			\textbf{Case $ i > 1 $:} We have
			$ \init(\alpha) = x_{\mathscr{F}}\otimes \bigwedge_{m=2}^{i} (\ell_{u_{m}} - \ell_{v_{m}}) 
			\in H_{i-1}(K(\mathbb{Q}[\overline{\Star}(\rho_{H})])) $.
			Then $ H $ covers $ \max(\mathscr{F}) $ and $ u_{2} \geq H \setminus \max(\mathscr{F}) $.

			Conversely, suppose we have a basis element $ x_{\mathscr{F}} \otimes \bigwedge_{m=2}^{i} (\ell_{u_{m}}-
			\ell_{v_{m}}) $  in $ H_{i-1}(K(\mathbb{Q}[\overline{\Star}(\rho_{H})])) $ with
			$ H $ covering $ \max(\mathscr{F}) $ and $ u_{2} \geq H \setminus \max(\mathscr{F}) $. Then
            define $ u_{1} = H \setminus \max(\mathscr{F}) $. We claim that
			$ \alpha = x_{\mathscr{F}} \otimes \bigwedge_{m=1}^{i}(\ell_{u_{m}} - \ell_{v_{m}}) $
			is a basis element of $ H_{\bullet}(K(\mathbb{Q}[\Sigma_{U_{r-1,n}}])) $ with
			$ \init(\alpha) = x_{\mathscr{F}} \otimes \bigwedge_{m=2}^{i} (\ell_{u_{m}}-
            \ell_{v_{m}}) $. To see that it is a basis element, note that $ \max(\mathscr{F}) $ is a hyperplane
            of $ U_{r-1,n} $, as it is covered by $H$, a hyperplane of $U_{r,n}$. As $ \mathscr{F} $ is retral in $ U_{r,n}^{H} $, 
            $ \mathscr{F} $ is weakly retral in $ U_{r-1,n} $. The wedge $ \bigwedge_{m=1}^{i}(\ell_{u_{m}} - \ell_{v_{m}}) $
            is admissible to $ \mathscr{F} $, as if $ \max(\mathscr{F}) $ covers a flat $F$ in $ \mathscr{F} $, then
            $ \max(\mathscr{F}) \setminus F > H \setminus \max(\mathscr{F}) = u_{1} $ because $ \mathscr{F} $ is retral
            in $U_{r,n}^{H} $. Finally, it is clear that
            $ \init(\alpha) = x_{\mathscr{F}} \otimes \bigwedge_{m=2}^{i} (\ell_{u_{m}}- \ell_{v_{m}}) $.

			\textbf{Case $ k =1 $:}
			We have  $ \init(\alpha) = x_{\mathscr{F}} \in H_{0}(K(\mathbb{Q}[\overline{\Star}(\rho_{H})])) $.
			Then $ H $ covers $ \max(\mathscr{F}) $, but it is not the maximal cover 
			of $ \max(\mathscr{F}) $ in $ \mathscr{L}(U_{r,n}) $. In particular,
			$ \max(\mathscr{F}) \cup v_{1} $ is a cover that is lexicographically greater than $H$.

			Conversely, suppose that $ x_{\mathscr{F}} $ is a basis element of 
			$ H_{0}(K(\mathbb{Q}[\overline{\Star}(\rho_{H})])) $ such that $H$ covers $ \max(\mathscr{F}) $
			and $ H $ is not the maximal cover of $ \max(\mathscr{F}) $ in $ \mathscr{L}(U_{r,n}) $.
			Let $ u_{1} = H \setminus \max(\mathscr{F}) $, and let $ v_{1} $ be the successor of 
			$ u_{1} $ in $ [n] \setminus \max(\mathscr{F}) $. This successor exists because $H$ is not 
			the maximal cover of $ \mathscr{F} $. We claim that $ \alpha = 
			x_{\mathscr{F}} \otimes (\ell_{u_{1}} - \ell_{v_{1}}) $ is a basis element of 
			$H_{1}(K(\mathbb{Q}[\Sigma_{U_{r-1,n}}])) $.
			It suffices to check that if $  \max(\mathscr{F}) $ covers $ F $ in $ \mathscr{F} $,
			then $ u_{1} < \max(\mathscr{F}) \setminus F $.
			Suppose this is the case.
			As $ \mathscr{F} $ is retral in $ \mathscr{L}(U_{r,n}^{H}) $,
			$ \max(\mathscr{F}) $ is the maximal cover of $ F $ in $ \mathscr{L}(U_{r,n}^{H}) $.
			This means that $ \max(\mathscr{F}) \setminus F > H \setminus \max(\mathscr{F}) = u_{1} $.
            Clearly, $ \init(\alpha) = x_{\mathscr{F}} \in H_{0}(K(\mathbb{Q}[\overline{\Star}(\rho_{H})])) $.
			\end{proof}

		We now deduce our main result.

		\begin{proof}[Proof of Theorem \ref{thm:basis}]
				Inducting on the rank $ r $, the base case is Lemma \ref{lem:baseCase}.
				By Lemma \ref{lem:initInjective}, the connecting homomorphism
				$ \delta_{i}\colon H_{i}(K(\mathbb{Q}[\Sigma_{U_{r-1,n}}])) 
						\to \bigoplus_{H} H_{i-1}(K(\mathbb{Q}[\overline{\Star}(\rho_{H})])) $
				is injective for $ i \geq 1 $. For $ i = 0 $, the connecting homomorphism
				vanishes. Therefore, 
				$ H_{\bullet}(K(\mathbb{Q}[\Sigma_{U_{r,n}}])) \cong
				H_{0}(K(\mathbb{Q}[\Sigma_{U_{r-1,n}}])) \oplus \bigoplus_{i \geq 1} \coker \delta_{i} $.
				To get a basis, we pull back the basis for $ H_{0}(K(\mathbb{Q}[\Sigma_{U_{r-1,n}}])) $
				and push forward the standard monomials from Lemma \ref{lem:standardMonomials}.
				One verifies that the basis elements from the pull-back are
					\[ \left\{ x_{\mathscr{F}} \mid \textrm{$ \mathscr{F} $ retral, not containing a hyperplane}
						\right\}, \]
				and the basis elements from the push-forward are
				\[ 
					\left\{ x_{\mathscr{F}} \mid \textrm{$ \mathscr{F} $ retral, containing a hyperplane}
						\right\} \cup 
					\left\{ x_{\mathscr{F}} \otimes \xi \mid \textrm{$ \mathscr{F} $ weakly
					retral and $ \xi $ admissible to $ \mathscr{F} $} \right\}. 
				\]
			Together, these elements form the desired basis.
		\end{proof}

		\begin{cor}\label{cor:basisChow}
			The set
			\[ 
				\left\{ x_{\mathscr{F}} : \textrm{$\mathscr{F} $ retral in $ \mathscr{L}(U_{r,n})$}
						\right\}
			\]
            is a basis for the Chow ring $ A^{\bullet}(X_{\Sigma_{U_{r,n}}}) $ as a $ \mathbb{Q}$-vector space.
		\end{cor}

        \begin{example}
            The Betti diagram for the  singular cohomology of $ U_{3,4} $ is 
            \vspace*{1em}
            \begin{center}
                \begin{tabular}{c|cc}
                     & 0 & 1 \\ \hline
                    0 & 1 & 0  \\
                    1 & 7 & 6 \\
                    2 & 1 & 3
                \end{tabular}
            \end{center}
            \vspace*{1em}
            The basis we have given for $ H_{\bullet}(K(\mathbb{Q}[\Sigma_{U_{3,4}}])) $ is
            \begin{center}
            \setlength{\tabcolsep}{8pt}
            \renewcommand{\arraystretch}{1.5}
                \begin{tabular}{c|cc}
                    & $0$ & $1$ \\ \hline
                    $0$ & $1$ & - \\  
                    $1$ & $ x_{4} $, $ x_{12} $, $ x_{13} $, $ x_{14} $, $ x_{23} $, $ x_{24} $, $ x_{34} $ &
                    \makecell{$ x_{12} \otimes (\ell_{3} - \ell_{4}) $, $ x_{13} \otimes (\ell_{2} - \ell_{4}) $,
                    $ x_{14} \otimes (\ell_{2} - \ell_{3}) $,\\ $ x_{23} \otimes (\ell_{1} - \ell_{4}) $, 
                    $ x_{24} \otimes (\ell_{1} - \ell_{3}) $, $ x_{34} \otimes (\ell_{1} - \ell_{2}) $} \\
                    $2$ & $ x_{4} x_{34} $ & $x_{3} x_{23} \otimes (\ell_{1} -\ell_{4}) $, $ x_{4} x_{24} \otimes (\ell_{1} - \ell_{3})$,
                    $ x_{4} x_{34} \otimes (\ell_{1} - \ell_{2}) $
                \end{tabular}
            \end{center}
            \vspace*{1em}
            arranged according to this Betti diagram.
        \end{example}

\section{Combinatorics of retral chains}\label{section:combinatorics}

    In this section, we study the combinatorics of retral and weakly retral chains in $ \mathscr{L}(U_{r,n}) $. From this
    we produce multiple formulas for the refined Hodge--Poincar\'{e} polynomials. When restricted to the
    special case of Chow polynomials, these formulas give combinatorial proofs for results of 
    Hoster \cite{hoster}, Hameister, Rao, and Simpson \cite{HRS}, and Ferroni et al. \cite{FMSV}.
    	
    \subsection{Rank-selected retral chains}
	We first count retral chains whose flats are of specified rank. This produces 
    a combinatorial proof of a formula for the Chow polynomial of a uniform matroid due to Hoster \cite{hoster}.

	\begin{definition}
		Let $ \mathscr{F} $ be a chain of proper flats of $ \mathscr{L}(U_{r,n}) $. 
		We define the \emph{ranks} of $ \mathscr{F} $ to be the subset
		$ \rks(\mathscr{F}) = \left\{ j \in \left\{0, \dots, r-1\right\} : 
		\textrm{$\mathscr{F} $ contains a flat of rank $j $} \right\} $.
        For $ i \in \rks(\mathscr{F}) $, we write $ F^{i} $ for the flat in $ \mathscr{F} $ of rank $i$.
	\end{definition}

	Recall our convention that every chain $ \mathscr{F} $ contains the empty flat.

	\begin{definition}
			Let $ R  \subseteq \left\{ 0, \dots, r-1 \right\} $ be a subset containing $0$.
			Define $ C_{R}(U_{r,n}) $ to be the set of retral chains $ \mathscr{F} $ in
			$ \mathscr{L}(U_{r,n}) $ such that 
			$ \rks(\mathscr{F}) = R $.
	\end{definition}

	For $ R \subseteq \left\{ 0, \dots, r-1 \right\} $,
	partition $ R = R_{1} \sqcup \cdots \sqcup R_{s} $ into its maximal consecutive subsets,
	ordered so that $ \min R_{i}  < \min R_{i+1} $. 
	Define the multinomial coefficient
		\[ 
				\binom{n}{\Delta R} = \binom{n}{\min R_{1} - 0, \min R_{2} - \min R_{1}, 
				\dots, \min R_{s} - \min R_{s-1}}
		\]
    and the index $ \lfloor R \rfloor = \min R_{s} $.

	One way to construct a retral chain $ \mathscr{F} \in C_{R}(U_{r,n}) $ is to choose an arbitrary
	chain of flats $ \mathscr{G} $ with $ \rks(\mathscr{G}) = \left\{ \min R_{1}, \dots, \min R_{s} \right\} $
	and then uniquely extend using the retral condition.

	\begin{lem}\label{lem:rankSelectedRetralConstruction}
		Let $ R $ be a subset of $ \left\{0, \dots, r-1 \right\} $ which contains $0$, and
		partition $ R = R_{1} \sqcup \cdots \sqcup R_{s} $ as above. 
		Given a chain of flats $ \mathscr{G} \subseteq \mathscr{L}(U_{r,n}) $ with
		$ \rks(\mathscr{G}) = \left\{ \min R_{1}, \dots, \min R_{s} \right\} $,
		there is a unique $ \mathscr{F} \in C_{R}(U_{r,n}) $ which extends $ \mathscr{G} $.
	\end{lem}

	\begin{proof}
        A chain of proper flats $ \mathscr{F} $ with $ \rks(\mathscr{F}) = R $ is retral if and only if each sequence of integers
		$ (F^{\min R_{j} + 1} \setminus F^{\min R_{j}}, \cdots, F^{\max R_{j}} \setminus F^{\max R_{j}-1}) $ is decreasing and 
		$ F^{\max R_{j}} \setminus F^{\max R_{j}-1} $ equals the maximal element
		of $ F^{\min R_{j+1}} \setminus F^{\max R_{j}-1} $ (here we use the convention that
		$ [n] = F^{\min R_{s+1}} $). This is satisfied if and only if $ (F^{\min R_{j} + 1} \setminus F^{\min R_{j}}, \cdots, 
		F^{\max R_{j}} \setminus F^{\max R_{j}-1}) $ is the sequence of the $ (\max R_{j} - \min R_{j}) $ largest
		elements of $ F^{\min R_{j+1}} \setminus F^{\min R_{j}} $ written in decreasing order. Therefore, there is
        exactly one such $ \mathscr{F} \in C_{R}(U_{r,n}) $ which extends $ \mathscr{G} $, as $ \mathscr{G} $ determines each
        $ F^{\min R_{i}} $.
	\end{proof}

		The number of chains $ \mathscr{G} \subseteq \mathscr{L}(U_{r,n}) $
        with $ \rks(\mathscr{G}) = \left\{ \min R_{1}, \dots, \min R_{s} \right\} $ is given by
		$ \binom{n}{\Delta R} $. As a consequence, we have the following count of rank-selected retral chains.
		
		\begin{prop}\label{prop:rankSelectedRetralChains}
			Let $ R $ be a subset of $ \left\{ 0, \dots, r-1 \right\} $ which contains
			$0$. Then $ \left| C_{R}(U_{r,n}) \right| = \binom{n}{\Delta R} $.
		\end{prop}

		The basis of Theorem \ref{thm:basis} thus gives a combinatorial proof, in terms of retral chains of flats, of 
        the following result of Hoster \cite[Theorem 3]{hoster}\footnote{Our indices differ by a shift from Hoster's.}.

		\begin{thm}
            For $ r \geq 1 $,
			\[ 
				\underline{\H}_{U_{r,n}}(x) = \sum_{\substack{R \subseteq \{0,\dots, r-1\} \\
						0 \in R}} \binom{n}{\Delta R} \cdot x^{\left| R \right| -1}.
			\]
		\end{thm}
		
	\subsection{Rank-selected weakly retral chains and admissible wedges}
		Our next goal is count the other basis elements
		$ x_{\mathscr{F}} \otimes \xi $, where $ \mathscr{F} $ is a weakly
		retral chain in $ \mathscr{L}(U_{r,n}) $ with prescribed ranks and $ \xi $ is admissible to 
		$ \mathscr{F} $.

		\begin{definition}
			Let $ R $ be a subset of $ \left\{0, \dots, r-1 \right\} $ which contains
			$ 0 $ and $ r- 1 $. Define $ W_{R, i}(U_{r,n}) $ to be the
			set of pairs $ (\mathscr{F}, \xi) $ such that $ \mathscr{F} $ is a weakly
			retral chain in $ \mathscr{L}(U_{r,n}) $, $ \rks(\mathscr{F}) = R $, and
			$ \xi \in \bigwedge^{i} N_{\mathbb{Q}}^{\vee} $  is admissible to $ \mathscr{F} $.
		\end{definition}

        Our method of constructing pairs $ (\mathscr{F}, \xi) \in W_{R,i}(U_{r,n}) $ is similar to our construction
        of elements of $ C_{R}(U_{r,n}) $ in that we will begin with an arbitrary chain $ \mathscr{G} $ with $ \rks(\mathscr{G}) = 
        \left\{ \min R_{1}, \dots, \min R_{s} \right\} $. However, this is not enough to uniquely determine
        $ (\mathscr{F}, \xi) $ from $ \mathscr{G} $, as there is some freedom in the choice of hyperplane of $ \mathscr{F} $, and
        $ \mathscr{G} $ imposes few restrictions on $ \xi = \bigwedge^{i}_{m=1} (\ell_{u_{m}} - \ell_{v_{m}})$. Our method of determining elements
        $ (\mathscr{F}, \xi) \in W_{R,i}(U_{r,n}) $ is outlined as follows:
        \begin{enumerate}
            \item Choose an arbitrary chain $ \mathscr{G} $ with $ \rks(\mathscr{G}) = \left\{ \min R_{1}, \dots, \min R_{s} \right\} $,
            \item choose $u_{1} $,
            \item choose the hyperplane of $ \mathscr{F} $,
            \item choose the remaining $ u_{2}, \dots, u_{i} $,
        \end{enumerate}
        then extend to a unique pair $ (\mathscr{F}, \xi) $.
        The choice in (1) will be arbitrary, but there will be some restrictions in (2)--(4).

        \begin{lem}\label{lem:rankSelectedWeaklyRetral}
			Let $R $ be a subset of $ \left\{ 0, \dots, r-1 \right\} $ which contains
			$ 0 $ and $ r-1 $. Partition $ R = R_{1} \sqcup \cdots \sqcup R_{s} $ as above. 		
			Given 
				\begin{enumerate}
					\item a chain of flats $ \mathscr{G} \subseteq \mathscr{L}(U_{r,n}) $ with $ \rks(\mathscr{G})
						  = \left\{ \min R_{1}, \dots, \min R_{s} \right\} $,
                    \item an integer $ 1 \leq \ell \leq n-r-i+1 $, 
                    \item an $ (r - \lfloor R \rfloor -1) $ element subset $F \subseteq [n] \setminus G^{\lfloor R \rfloor} $ 
                            with $ \min F $ greater than the $ \ell$th smallest element of $[n] \setminus G^{\lfloor R \rfloor} $,
                        \item an $ (i-1) $-element subset $I \subseteq [n] \setminus (G^{\lfloor R \rfloor} \cup F)$  with
                            $ \min I $ greater than the $\ell$th smallest element of $ [n] \setminus G^{\lfloor R \rfloor} $ and
                            $ \max I < \max [n] \setminus (G^{\lfloor R \rfloor} \cup F)$,
				\end{enumerate}
                there is a unique pair $ (\mathscr{F}, \xi= \bigwedge_{m=1}^{i} (\ell_{u_{m}} - \ell_{v_{m}})) \in W_{R,i}(U_{r,n}) $ such that
                \begin{enumerate}[label = (\alph*)]
                                \item $ \max(\mathscr{F}) = F \cup G^{\lfloor R \rfloor} $,
								\item $ \mathscr{F} $ extends $ \mathscr{G} $,
                                \item $ u_{1} $ is the $\ell$th smallest element of $[n] \setminus G^{\lfloor R \rfloor} $,
                                \item $ I = \left\{u_{2}, \dots, u_{i} \right\} $.
				\end{enumerate}
		\end{lem}

		\begin{proof}
            Similar to the proof of Lemma \ref{lem:rankSelectedRetralConstruction}, there is a unique weakly retral chain
            $ \mathscr{F} $ with $ \rks(\mathscr{F}) = R $ satisfying conditions (a) and (b). Conditions (c) and (d) uniquely determine
            $ \xi = \bigwedge_{m=1}^{i} (\ell_{u_{m}}- \ell_{v_{m}}) $, as we must have $ u_{1} < u_{2} < \cdots < u_{i} $
            and $ v_{m} $ must be the successor of 
            $ u_{m} $ in $ [n] \setminus \max(\mathscr{F}) $. Note that the successor of $ u_{i} $ exists because of the restriction
            in (4).

            It remains to see that $ \xi $ is admissible to $ \mathscr{F} $. Indeed, if $ F^{r-1} $ covers $ F^{r-2} $ in $ \mathscr{F} $,
            then $ u_{1} < F^{r-1} \setminus F^{r-2} $ due to the restriction in (3) and condition (c).
		\end{proof}

        It is straightforward to see that every pair $(\mathscr{F}, \xi) \in W_{R,i}(U_{r,n}) $ arises from Lemma \ref{lem:rankSelectedWeaklyRetral}
        for a unique set of data $ \mathscr{G}, \ell, F, I $. 
        The following counts the possible choices of $ \mathscr{G} $, $ \ell $, $F$, and $I$ above.

		\begin{prop}\label{prop:rankSelectedWeaklyRetral}
			Let $R$ be a subset of $ \left\{0, \dots, r-1 \right\} $ containing $ 0$ and $ r-1 $.
            Then 
            \[ 
                \left| W_{R}(U_{r,n}) \right| = \binom{n}{\Delta R} \, \sum_{\ell=1}^{n-r-i+1} \, \binom{n - \lfloor R \rfloor - \ell}{r-
                \lfloor R \rfloor - 1} \binom{n-r-\ell}{i-1}.
            \]
		\end{prop}

		Summing over appropriate subsets $ R $  produces the refined
		Hodge--Poincar\'{e} polynomials.

        \begin{thm}\label{thm:hosterHP}
			For $ 1 \leq i \leq n-r $, 
			\[ 
						\underline{\H}^{i}_{U_{r,n}}(x) = \sum_{\substack{ R \subseteq 
						\left\{ 0,\dots, r-1 \right\} \\ 0, r-1 \in R}}
                        \sum_{\ell=1}^{n-r-i+1} \, \binom{n}{\Delta R} \binom{n - \lfloor R \rfloor - \ell}{r-\lfloor R \rfloor -1} 
                        \binom{n-r-\ell}{i-1} \cdot x^{\left| R \right| -1}.
			\]
		\end{thm}

		\subsection{Refined Hodge--Poincar\'{e} polynomials and derangement polynomials}
		In this subsection, we give another way of constructing retral chains, this time through concatenation. This leads to combinatorial proofs of
		results from \cite{HRS} and \cite{FMSV} which express the Chow polynomial $\underline{\H}_{U_{r,n}}(x) $ in
		terms of the \emph{derangement polynomials}. We then obtain similar formulas for the other 
		refined Hodge--Poincar\'{e} polynomials.

		\begin{definition}
			Let $ \mathfrak{D}_{n} $ denote the set of \emph{derangements} on $[n]$---these
			are the permutations in $ \mathfrak{S}_{n} $ with no fixed points. Recall
            that an \emph{excedence} of $ \sigma \in \mathfrak{S}_{n} $ is an index $ i \in [n] $ such that 
		    $ i < \sigma(i) $. We write $ \exc(\sigma) $ for the number of excedences of $ \sigma $.

			For $ n \geq 1 $, the $n$th \emph{derangement polynomial} is
			\[ 
					d_{n}(x) = \sum_{\sigma \in \mathfrak{D}_{n}} x^{\exc(\sigma)},
			\]
			with $ d_{0}(x) = 1 $ by convention.
		\end{definition}

        For a chain of flats $ \mathscr{F} $, we define $ \mathscr{F}_{\geq j} = \left\{ F^{i} : i \in \rks(\mathscr{F}) \textrm{ and }
        i \geq j \right\} $ and $ \mathscr{F}_{<j} $ similarly.

		\begin{lem}\label{lem:retralConstruction}
			Let $ j > 1 $ and suppose $ \mathscr{F} $ is a retral chain in $ \mathscr{L}(U_{r,n}) $ with
			$ j \in \rks(\mathscr{F}) $ but $ j-1 \notin \rks(\mathscr{F}) $. Then
            $ \emptyset < \mathscr{F}_{\geq j} $ is a retral chain in $ \mathscr{L}(U_{r,n}) $
            with $ \min \rks(\mathscr{F}_{\geq j}) = j $, and
			$ \mathscr{F}_{<j} $ is a
            retral chain in $ \mathscr{L}(U_{r,n}^{F^{j}}) $ with $ \max \rks(\mathscr{F}_{<j}) < j-1 $.

			The converse is also true: If $ \emptyset  < \mathscr{F}_{\geq j} $ is a retral chain
			in $ \mathscr{L}(U_{r,n}) $ with $ \min \rks(\mathscr{F}_{\geq j}) = j $
			and $ \mathscr{F}_{<j} $ is a retral chain in $ \mathscr{L}(U_{r,n}^{F^{j}}) $ 
			with $ \max \rks(\mathscr{F}_{<j}) < j-1 $, then the concatenation 
			$ \mathscr{F}_{< j} < \mathscr{F}_{\geq j} $ is a retral chain in
			$ \mathscr{L}(U_{r,n}) $.
		\end{lem}

		\begin{proof}
			The first statement is immediate from the definition of retral chains.
			The second statement follows because $ F^{j} $ covers no flat in
			$ \mathscr{F}_{<j} < \mathscr{F}_{\geq j} $, so the maximal cover condition on this flat is vacuous.
		\end{proof}

		We use this construction to count retral chains in $ \mathscr{L}(U_{r,n}) $, yielding a
		formula for $ \underline{\H}_{U_{r,n}}(x) $. For the following, we use the convention that $ \underline{\H}_{U_{0,1}}(x) = 0 $
		and $ \underline{\H}_{U_{-1,0}}(x) = \frac{1}{x} $. 

		\begin{prop}\label{prop:recursiveCorank1}
			For $ r \geq 1 $, 
			\[ 
				\underline{\H}_{U_{r,n}}(x) = \sum_{j=0}^{r-1}\, \binom{n}{j} \, \underline{\H}_{U_{j-1,j}}(x)
						\cdot (x + x^{2} + \cdots + x^{r-j}).
			\]
		\end{prop}

		\begin{proof}
			For $ j \geq 0 $, we claim that the generating function, with respect to length, 
            for retral chains $ \mathscr{F} $ in $ \mathscr{L}(U_{r,n}) $  
			such that $j$ is the maximal index with $ j \in \rks(\mathscr{F}) $ and $ j-1 \notin \rks(\mathscr{F}) $ is
			\[ 
				\binom{n}{j} \, \underline{\H}_{U_{j-1,j}}(x) \cdot (x + x^{2} + \cdots + x^{r-j}).
			\]
			We begin with the edge cases of $ j=0, 1 $. For $j=0 $, Lemma \ref{lem:rankSelectedRetralConstruction}
			implies that such retral chains are determined by their length, which ranges from $0$ to $ r-1 $. 
			The generating function is then $ 1 + \cdots + x^{r-1}$, which agrees with the claim by our convention 
            for $ \underline{\H}_{U_{-1,0}}(x) $.
			For $j=1 $, our convention that all retral chains contain $ \emptyset $ implies the generating function is $0$
            which also agrees with the claim by our convention for $ \underline{\H}_{U_{0,1}}(x) $.

			For $ j > 1 $, Lemma \ref{lem:retralConstruction} implies that such $ \mathscr{F} $ is determined uniquely by
			$ \mathscr{F}_{\geq j} $ and $ \mathscr{F}_{<j} $. Note that $ \rks(\mathscr{F}_{\geq j}) $
			is a consecutive sequence of integers, so  Lemma \ref{lem:rankSelectedRetralConstruction} implies that
			it is determined by $ F^{j} $ and its length from $ 1 $ to $ r-j $. On the other hand,
			$ \mathscr{F}_{<j} $ is a retral chain of $ \mathscr{L}(U_{r,n}^{F^{j}}) $ 
			with $ \max \rks(\mathscr{F}_{<j}) < j-1 $. This is equivalent to a retral chain in
			$ \mathscr{L}(U_{j-1,j}) $. Therefore, the generating function is as claimed, and the result follows by summing over $j$.
		\end{proof}

		We now relate the Chow polynomials of uniform matroids to the derangement polynomials.

		\begin{prop}[{\cite[Theorem 5.1]{HRS} and \cite[Theorem 1.9]{FMSV}}] \label{prop:HRS}
			The following formulas for the Chow polynomials hold:
			\begin{align*}
						x \cdot \underline{\H}_{U_{n-1,n}}(x) &= d_{n}(x) \\
						\underline{\H}_{U_{r,n}}(x) &= \sum_{j=0}^{r-1} \, \binom{n}{j} \, d_{j}(x) \cdot
			 (1 + x + \cdots + x^{r-1-j}).
			\end{align*}
		\end{prop}

		\begin{proof}
			Proposition \ref{prop:recursiveCorank1} applied to $ U_{n-1,n} $ shows that
            $ \{ x \cdot \underline{\H}_{U_{n-1,n}}(x) \}_{n=0}^{\infty} $  satisfies the recurrence
            \[ 
                f_{n}(x) = \sum_{j=0}^{n-2} \, \binom{n}{j} \, f_{j}(x) \cdot (x + \cdots + x^{n-1-j})
            \]
            which the family $ \{ d_{n}(x) \}_{n=0}^{\infty} $ also satisfies by \cite[Corollary 4.2]{JKMS}. 
            As $ x \cdot \underline{\H}_{U_{n-1,n}}(x) =
			d_{n}(x) $ for $ n=0,1 $, we have $ x \cdot \underline{\H}_{U_{n-1,n}}(x) = d_{n}(x) $ 
			for all $n$. 

			Proposition \ref{prop:recursiveCorank1} yields the formula for 
			$ \underline{\H}_{U_{r,n}}(x) $ after the substitution 
			$ d_{j}(x) = x \cdot \underline{\H}_{U_{j-1,j}}(x) $.
		\end{proof}

		\subsubsection{Weakly retral chains and admissible wedges}

		A similar strategy of concatenation produces a method for constructing pairs $ (\mathscr{F}, \xi) $
		of a weakly retral chain $ \mathscr{F} $ and admissible wedge $ \xi $.

		\begin{lem}\label{lem:constructionWeaklyRetralConcatenation}
			Let $ j > 1 $ and suppose that $ \mathscr{F} $ is a weakly retral chain in
			$ \mathscr{L}(U_{r,n}) $ with $ j \in \rks(\mathscr{F}) $ but $ j-1 \notin
			\rks(\mathscr{F}) $. Suppose also that $ \xi $ is admissible to $ \mathscr{F} $.
			Then $ \emptyset < \mathscr{F}_{\geq j} $ is a weakly retral chain in $ \mathscr{L}(U_{r,n}) $
            with $\min \rks(\mathscr{F}_{\geq j}) = j $,
			$ \mathscr{F}_{<j} $ is a retral chain in $ \mathscr{L}(U_{r,n}^{F^{j}}) $ with
			$ \max \rks(\mathscr{F}_{<j}) < j-1 $, and $ \xi $ is admissible to $ \emptyset < \mathscr{F}_{\geq j} $.

			The converse also holds.
		\end{lem}

		From this construction we derive a formula for the refined Hodge--Poincar\'{e} polynomials 
		in terms of the derangement polynomials.

        \begin{thm}\label{thm:derangementHP}
			For $ 1 \leq i \leq n-r $,
			\[ 
				\underline{\H}^{i}_{U_{r,n}}(x) = \sum_{j=0}^{r-1}
                \sum_{\ell=1}^{n-r-i+1} \, \binom{n}{j} \binom{n - j- \ell}{r-j - 1} \binom{n-r-\ell}{i-1} \, d_{j}(x) \cdot x^{r-j-1}.
			\]
		\end{thm}

		\begin{proof}
			For $j \geq 0 $, we claim the generating function, with respect to length of chain, for pairs $ (\mathscr{F}, \xi) 
            \in \bigcup_{R} W_{R,i}(U_{r,n}) $ such that $j$ is the maximal integer with $ j \in \rks(\mathscr{F}) $ 
			and $ j-1 \notin \rks(\mathscr{F}) $ is
			\[ 
                \sum_{\ell=1}^{n-r-i+1} \, \binom{n}{j} \binom{n - j- \ell}{r-j - 1} \binom{n-r-\ell}{i-1} \, d_{j}(x) \cdot x^{r-j-1}.
			\]
			Indeed, Lemma \ref{lem:constructionWeaklyRetralConcatenation} implies that such pairs are uniquely
			determined by $ (\emptyset < \mathscr{F}_{\geq j}, \xi) \in 
            W_{R, i}(U_{r,n}) $ with $ R = \left\{0, j, j+1, \dots, r-1 \right\} $ and a retral chain
			$ \mathscr{F}_{<j} \in \mathscr{L}\left( U_{r,n}^{F^{j}_{\geq j}}\right) $ with $ j-1 \notin 
                \rks(\mathscr{F}_{<j}) $. 
			By Proposition \ref{prop:rankSelectedWeaklyRetral}, the generating function for the former is
			\[
                \sum_{\ell=1}^{n-r-i+1} \, \binom{n}{j} \binom{n - j- \ell}{r-
                j - 1} \binom{n-r-\ell}{i-1} \cdot x^{r-j},
			\]
            while the generating function for the latter is $d_{j}(x) \cdot \frac{1}{x} $ by Proposition \ref{prop:HRS}. 
			The result follows by summing over $j$.
		\end{proof}

    \subsection{Recursive formulas for refined Hodge--Poincar\'{e} polynomials}
        The truncation exact sequence \eqref{eq:truncationLES} and 
        Lemma \ref{lem:initInjective} provide a simple recurrence for the singular cohomology
        of uniform matroids. In this subsection, we use this recurrence to produce closed formulas for the refined Hodge--Poincar\'{e}
        polynomials.

        Recall the truncation exact sequence
        \begin{equation*}
            \cdots \to \bigoplus_{H \in \mathscr{H}(U_{r,n})} 
			H_{i}(K(\mathbb{Q}[\overline{\Star}(\rho_{H})] ) )_{j}
			\to H_{i} ( K ( \mathbb{Q}[\Sigma_{U_{r,n}}] ) )_{j+1} \to 
			H_{i}( K ( \mathbb{Q}[\Sigma_{U_{r-1,n}}] ) )_{j+1} \to \cdots
		\end{equation*}
        with connecting homomorphism 
        \[ 
            \delta_{i} \colon H_{i}( K(\mathbb{Q}[\Sigma_{U_{r-1,n}}] ))_{j}  \to
			\bigoplus_{H} H_{i-1}(K(\mathbb{Q}[\overline{\Star}(\rho_{H})]))_{j-1}.
	 	\]
        Lemma \ref{lem:initInjective} shows that $ \delta_{i} $ is injective for $ i \geq 1 $,
        while $ \delta_{0} $ vanishes trivially. 

        In order to get a recurrence for the refined Hodge--Poincar\'{e} polynomials, we first express the dimensions
        of $ H_{i}(K(\mathbb{Q}[\overline{\Star}(\rho_{H})]))_{j} $ in terms of \emph{Eulerian polynomials}.

        \begin{definition}
            The \emph{$n$th Eulerian polynomial} $ A_{n}(x) $ is the generating function for permutations in $ \mathfrak{S}_{n} $ with respect to 
            the number of excedences:
            \[ 
                A_{n}(x) = \sum_{\sigma \in \mathfrak{S}_{n}} x^{\exc(\sigma)},
            \]
            and we set $ A_{0}(x) =1 $ by convention.
        \end{definition}

        \begin{lem}
            For $ i \geq 0 $,
            \[ 
                \sum_{H \in \mathscr{H}(U_{r,n})} \sum_{j=0}^{r-2} \, \dim H_{i}(K(\mathbb{Q}[\overline{\Star}(\rho_{H})]))_{j} \cdot x^{j} 
                = \, \binom{n}{r-1}\binom{n-r}{i} \, A_{r-1}(x).
            \]
        \end{lem}

        \begin{proof}
            Lemma \ref{lem:hyperplaneBasis} implies that \[ \sum_{j=0}^{r-2} \, \dim H_{i}(K(\mathbb{Q}[\overline{\Star}(\rho_{H})]))_{j} \cdot x^{j} 
             = \binom{n-r}{i} \, \underline{\H}_{U_{r-1,r-1}}(x), \]
            for each $ H \in \mathscr{H}(U_{r,n}) $. It is well-known that 
            $ \underline{\H}_{U_{r-1,r-1}}(x) = A_{r-1}(x) $ (see \cite[Theorem 5.1]{HRS}). As
            $ \left| \mathscr{H}(U_{r,n})\right| = \binom{n}{r-1} $, the lemma follows.
        \end{proof}
    
        For $ i \geq 1 $, the truncation exact sequence therefore provides the recurrence
		\[ 
				\underline{\H}_{U_{r,n}}^{i}(x) = \binom{n}{r-1} \binom{n-r}{i}\, A_{r-1}(x) \cdot x- \underline{\H}_{U_{r-1, n}}^{i+1}(x)
                \cdot x,
		\]
        which produces the following formula.

        \begin{prop}\label{prop:recursiveRHP}
			For $ i \geq 1 $,
			\[
                \underline{\H}_{U_{r,n}}^{i}(x) = (-1)^{r-1}\, \binom{n-1}{i+r-1} \cdot x^{r-1} 
						+ \sum_{j=1}^{r-1}(-1)^{r+1-j} \, \binom{n}{j}\binom{n-j-1}{i+r-j-1} \,  A_{j}(x)\cdot x^{r-j}.
			\] 
		\end{prop}

	\begin{proof}
		The base case is $ \underline{\H}^{i}_{U_{1,n}}(x) = \binom{n-1}{i} $, and the result follows from induction on $r$.
	\end{proof}

        The case $ i = 0 $ is for the Chow polynomial, and we have the recurrence
        \[ 
				\underline{\H}_{U_{r,n}}(x) = \binom{n}{r-1} \, A_{r-1}(x) \cdot x + 
					\underline{\H}_{U_{r-1,n}}(x) -\underline{\H}^{1}_{U_{r-1,n}}(x) \cdot x,
        \]
        which recovers the following formula of Ferroni et al.

        \begin{prop}[{\cite[Equation (13)]{FMSV}}]\label{prop:recursiveChow}
		For $ r \geq 1 $,
		\[ 
					\underline{\H}_{U_{r,n}}(x) = \sum_{j=0}^{r-1} \sum_{m=0}^{r-j-1} (-1)^{m}  \,
					\binom{n}{j} \binom{n-j-1}{m} \, A_{j}(x) \cdot x^{r - m -j -1}.
		\]	
	\end{prop}
	
	\begin{proof}
        By induction and Proposition \ref{prop:recursiveRHP}, one can show that
		\[
			\underline{\H}_{U_{r,n}}(x) = \sum_{m=0}^{r-1} (-1)^{m} \, \binom{n-1}{m} \cdot x^{m} 
					+ \sum_{j=1}^{r-1} \sum_{m=0}^{r-j-1} (-1)^{m} \, \binom{n}{j} \binom{n-j-1}{m}\,  A_{j}(x) \cdot x^{m+1}.
		\]
		The desired equation follows by taking the reverse polynomial of each side and using the fact
		$ \underline{\H}_{U_{k,n}}(x) $ and $ A_{j}(x) $ are palindromic.
	\end{proof}

    \begin{example}
            We use these formulas to list the Betti diagrams for the singular cohomology of loopless uniform matroids on ground set $[5]$
            in Table \ref{table:groundset}.
            \vspace*{1em}
            \begin{table}
                \setlength{\tabcolsep}{6pt}
                \renewcommand{\arraystretch}{1}
                \hfill
                \begin{tabular}{c|ccccc}
                    $U_{1,5} $ & $0$& $1$& $2$ & $3$ & $4$ \\ \hline 
                    $0$ & $1$ & $4 $& $6 $ & $4$ & $1$  \\
                    \multicolumn{6}{c}{} \\
                    \multicolumn{6}{c}{}
                \end{tabular}
                \hfill
                \begin{tabular}{c|cccc}
                    $U_{2,5} $ & $ 0 $ & $1$ & $2$ & $ 3 $ \\ \hline
                    $0$ & $1$ & $ 0 $ & $0$ & $0$ \\
                    $ 1$ & $ 1 $ & $ 9 $ & $ 11 $ & $ 4 $  \\
                    \multicolumn{5}{c}{}
                \end{tabular} 
                \hfill
                \begin{tabular}{c|ccc}
                    $U_{3,5} $ & $ 0 $ & $ 1 $ & $ 2 $ \\ \hline
                    $0 $ & $1$ & $0$ & $0$ \\
                    $ 1$ & $ 11 $ & $ 20 $ & $ 10 $ \\
                    $ 2 $ & $ 1 $ & $9$ & $6$
                \end{tabular} 
                \hfill \

                \vspace*{1em}

                \hfill
                \begin{tabular}{c|cc}
                    $U_{4,5} $ & $ 0 $ & $ 1 $  \\ \hline
                    $ 0 $ & $ 1 $ & $0$ \\
                    $ 1 $ & $ 21 $ & $ 10 $ \\
                    $ 2 $ & $ 21 $ & $ 30 $ \\
                    $ 3 $ & $ 1 $ & $ 4 $ \\
                    \multicolumn{3}{c}{}
                \end{tabular}
                \hfill
                \begin{tabular}{c|c}
                    $U_{5,5} $ & $ 0 $ \\  \hline
                    $ 0 $ & $ 1 $ \\
                    $ 1 $ & $ 26 $ \\
                    $ 2 $ & $ 66 $ \\
                    $ 3 $ & $ 26 $ \\
                    $ 4 $ & $ 1 $
                \end{tabular}
                \hfill \phantom{x}

                \caption{Betti diagrams for the singular cohomology of loopless uniform matroids on ground set $[5]$.}\label{table:groundset}
            \end{table}

        \end{example}

	\section{Lefschetz properties}\label{section:lefschetz}
        In this final section, we show that the singular cohomology ring of a uniform
        matroid satisfies a slight weakening of Hard Lefschetz which we call the \emph{quasi-projective strong Lefschetz property}.

		\subsection{Quasi-projective Lefschetz properties}
			We begin by recalling the weak and strong Lefschetz properties for standard
			graded Artinian algebras. See \cite{miglioreNagel} for a survey.

			\begin{definition}
				Let $ A^{\bullet} $ be an Artinian graded algebra and $ \ell \in A^{1} $ 
				a generic linear form.
				\begin{enumerate}
					\item We say that $A^{\bullet} $ satisfies the \emph{weak Lefschetz property} (WLP) if 
									\[ 
											\times \ell \colon A^{i} \longrightarrow A^{i+1}
									\]
                                    has maximal rank for all $ i \in \mathbb{Z}_{\geq 0} $.

					\item We say that $ A^{\bullet} $ satisfies the \emph{strong Lefschetz property} (SLP)
								if 
								\[ 
												\times  \ell^{d}\colon A^{i} \longrightarrow A^{i+d}
								\]
								has maximal rank for all $ i, d \in \mathbb{Z}_{\geq 0} $.
				\end{enumerate}
			\end{definition}

			For a smooth, quasi-projective variety $X$, the singular cohomology ring
			$ H^{\bullet}(X) $ is Artinian, so one can ask if it satisfies 
			WLP or SLP. However, this will often fail when torus factors are involved.

            \begin{example}\label{ex:sqSLP}
				Let $ Y $ be a smooth, projective toric variety of (complex) dimension $s$, and
				consider $ X = Y \times (\mathbb{C}^{*})^{t} $.

                Hodge theory implies that for
				a generic $ \ell \in \Gr_{2}^{W}H^{2}(Y) $,  the map
				\[ 
					\times \ell^{d} \colon H^{k}(Y) \longrightarrow H^{k + 2d}(Y)
				\]
				is injective for $ d \leq s - k $ and surjective for $ d \geq s-k $.
							Therefore, $ H^{\bullet}(Y) $ satisfies SLP.

				However, $ H^{\bullet}(X) $ does not satisfy SLP when $ t \geq 2 $.
				Indeed, $ H^{k}(X) \cong \bigoplus_{i=0}^{t} H^{k-i}(Y)^{\oplus \binom{t}{i}} $,
				so $ \times  \ell^{s -k + 1} \colon H^{k}(X) \to H^{2s -k + 2}(X) $ is 
				neither injective on the subfactor $ H^{k}(Y) \to  H^{2s-k+2}(Y) $ nor surjective
				on the subfactor $ H^{k-2}(Y)^{\oplus \binom{t}{2}} \to  H^{2s - k}(Y)^{\oplus \binom{t}{2}} $
                for $ 2 \leq k \leq s $.

				The weight filtration rectifies the situation. Now, 
				$ \Gr_{j}^{W} H^{k}(X) \cong H^{2k-j}(Y)^{\oplus \binom{t}{j-k}} $,
				so
				\[
						\times \ell^{d} \colon \Gr_{j}^{W} H^{k}(X) \to 
							\Gr_{j + 2d}^{W} H^{k + 2d} (X) 
				\]
				is injective for $ d \leq s + j - 2k  $ and surjective for $ d \geq s + j -2k $.
			\end{example}

            This example motivates using the weight filtration to define Lefschetz properties for quasi-projective varieties.

			\begin{definition}
				Let $ X $ be a smooth, quasi-projective variety and  $ \ell \in
				\Gr_{2}^{W}H^{2}(X) $ a generic hyperplane class in the singular cohomology ring.
				\begin{enumerate}
					\item We say that $H^{\bullet}(X) $ satisfies the
						\emph{quasi-projective weak Lefschetz property} (qpWLP) if								
                            \[ 
							    \times \ell \colon \Gr_{j}^{W}H^{k}(X) \longrightarrow
								\Gr_{j+2}^{W}H^{k+2}(X)
							\]
					    has full rank for all $j,k \in \mathbb{Z}_{\geq 0} $.

					\item We say that $ H^{\bullet}(X) $ satisfies the \emph{quasi-projective
								strong Lefschetz property} (qpSLP) if 
								\[ 
										\times \ell^{d} \colon \Gr_{j}^{W}H^{k}(X) \longrightarrow
												\Gr_{j+2d}^{W}H^{k+2d}(X)
								\]
								has full rank for all $ d,j,k \in \mathbb{Z}_{\geq 0} $.
				\end{enumerate}
			\end{definition}
			
		\subsection{Lefschetz in the singular cohomology ring of a
		uniform matroid}
			
		We now show that the singular cohomology ring of a uniform matroid satisfies the quasi-projective
        strong Lefschetz property. 

        Our proof uses the truncation exact sequence \eqref{eq:truncationLES} to bootstrap
        Lefschetz properties for the singular cohomology of $ U_{r,n} $ from the singular cohomology
        of  the permutohedral variety.
	
        \begin{lem}\label{lem:lefschetzPermutohedral}
            For a hyperplane $H \in \mathscr{L}(U_{r,n}) $ and
            generic  $ \ell \in \Gr_{2}^{W}H^{2}(X_{\overline{\Star}(\rho_{H})}) $,
			the map
			\[ 
                    \times \ell^{d}	\colon \Gr_{j}^{W}H^{k}\left( X_{\overline{\Star}(\rho_{H})} \right)
                    \to \Gr_{j+2d}^{W}H^{k+2d} \left( X_{\overline{\Star}(\rho_{H})} \right)
			\]
			is injective for $ 0 \leq d \leq r +j -2k -2 $ and surjective for $ d \geq r + j -2k -2 $.
	    \end{lem}

        \begin{proof}
            By \cite[Lemma A.3]{binder}, $ H^{\bullet} \left( X_{\overline{\Star}(\rho_{H})}\right)
            \cong H^{\bullet} \left( X_{\Star(\rho_{H})} \right) $, so it suffices to
            consider $ X_{\Star(\rho_{H})} $. Now, $ X_{\Star(\rho_{H})} \cong X_{\Sigma_{U_{r-1,r-1}}} 
            \times (\mathbb{C}^{*})^{n-\left| H \right|-1} $.
            As $ X_{\Sigma_{U_{r-1,r-1}}} $ is the permutohedral variety of dimension $r-2$, 
            the result follows as in Example \ref{ex:sqSLP}.
        \end{proof}

        Now consider an arbitrary class $ \ell \in \Gr_{2}^{W}H^{2}(X_{\Sigma_{U_{r,n}}}) $, and let
        $ \ell_{H} $ denote its pull-back to $ X_{\overline{\Star}(\rho_{H})} $ and
        $ \bar{\ell} $ its pull-back to $ X_{\Sigma_{U_{r-1,n}}} $. 
        In terms of linear forms in the Stanley--Reisner ring, $ \ell_{H} $ quotients out the indeterminates
        $ x_{F} $ such that $ F \not\leq H $ and $ \bar{\ell} $ quotients out the indeterminates $ x_{H} $
        for $ H \in \mathscr{H}(U_{r,n}) $.
        The truncation exact sequence \eqref{eq:truncationLES} produces commuting squares
        \begin{center}
            \begin{tikzcd}
            \Gr_{j}^{W} H^{k}(\Sigma_{U_{r-1,n}}) \dar{\times \bar{\ell}^{d}} \rar{\delta} &
            \displaystyle \bigoplus_{H \in \mathscr{H}(U_{r,n})} \Gr_{j-2}^{W}H^{k-1}\left( 
                X_{\overline{\Star}(\rho_{H})} \right)
            \dar{\oplus \times \ell_{H}^{d}} \\
            \Gr_{j+2d}^{W} H^{k+2d}(\Sigma_{U_{r-1,n}})  \rar{\delta} &
            \displaystyle \bigoplus_{H \in \mathscr{H}(U_{r,n})} \Gr_{j + 2d -2}^{W}H^{k + 2d -1}\left( 
                X_{\overline{\Star}(\rho_{H})} \right)
            \end{tikzcd}
        \end{center}
        and
    	\begin{center}
	    	\begin{tikzcd}
                \displaystyle \bigoplus_{H \in \mathscr{H}(U_{r,n})} 
                    \Gr_{j-2}^{W}H^{k-2}\left( X_{\overline{\Star}(\rho_{H})} \right)
                        \dar{\oplus \times \ell_{H}^{d}}
                        \rar &  \Gr_{j}^{W} H^{k} \left(X_{\Sigma_{U_{r,n}}} \right)  \dar{\times \ell^{d}} \\
                        \displaystyle \bigoplus_{H \in \mathscr{H}(U_{r,n})} 
                        \Gr_{j+2d-2}^{W}H^{k+2d -2}\left( X_{\overline{\Star}(\rho_{H})} \right) \rar
                        & \Gr_{j+ 2d}^{W} H^{k +2d} \left(X_{\Sigma_{U_{r,n}}} \right).
		    \end{tikzcd}
		\end{center}

	\begin{thm}\label{thm:qpSLP}
        Let $ \ell \in \Gr_{2}^{W}H^{2}(X_{\Sigma_{U_{r,n}}}) $ be generic. Then the map
        \[ 
            \times \ell^{d} \colon \Gr_{j}^{W}H^{k}(X_{\Sigma_{U_{r,n}}}) \to
            \Gr_{j+2d}^{W}H^{k+2d}(X_{\Sigma_{U_{r,n}}})
        \]
        is injective for $ 0 \leq d \leq r + j-2k-1 $ and surjective for $ d > r + j-2k-1 $.
        In particular, the singular cohomology ring of a uniform matroid satisfies qpSLP.
	\end{thm}

    \begin{proof}
        If $ j=k $ or $ r= n $, this is a weaker statement than Hard Lefschetz for the Chow ring
        of a matroid or the permutohedral variety, respectively.
        Therefore, assume $ j \neq k $ and $ r < n $. 

        When $ 0 \leq d \leq r + j-2k- 1$, we use the first commuting square from above. 
        We may find $ \hat{\ell} \in \Gr_{2}^{W}H^{k}(X_{\Sigma_{U_{r+1,n}}}) $ such that
        $ \ell = \bar{\hat{\ell}} $. By genericity, we may also assume each $ \hat{\ell}_{H} $ is generic.
        The map $ \oplus \times \hat{\ell}_{H}^{d} $ is injective by Lemma \ref{lem:lefschetzPermutohedral}.
        The horizontal maps are also injective by Lemma \ref{lem:initInjective}, so $ \times \ell^{d} $ 
        is injective as well.

        When $ d > r + j-2k - 1 $, we use the second square. The map $ \oplus \times \ell_{H}^{d} $ 
        is surjective by Lemma \ref{lem:lefschetzPermutohedral}, and the  horizontal maps are surjective 
        by Lemma \ref{lem:initInjective}. This implies $ \times \ell^{d} $ is surjective.
    \end{proof}

    \begin{rmk}
        This proof can be adapted to show that the conclusion of Theorem \ref{thm:qpSLP} 
        holds for any $ \ell \in \Gr_{2}^{W}H^{2}(X_{\Sigma_{U_{r,n}}}) $ which is the pull-back of a class $ \hat{\ell} $ from the
        \emph{ample cone} of $ \Sigma_{U_{r+1,n}} $ (see \cite[Definition 4.2]{AHK}).
        The idea is that both $ \hat{\ell}_{H} $ and $ \bar{\hat{\ell}}_{H} $ are in the ample cone
        of $ \Star(\rho_{H}) $ for each $ H \in \mathscr{H}(U_{r+1,n}) $ and $H \in \mathscr{H}(U_{r,n}) $, respectively,
        by \cite[Proposition 4.4]{AHK}. Then \cite[Theorem 8.8]{AHK} implies the conclusion of Lemma \ref{lem:lefschetzPermutohedral}
        holds for $ \hat{\ell}_{H} $ and $ \bar{\hat{\ell}}_{H} $, and the rest of the proof follows verbatim.
    \end{rmk}

    A combinatorial consequence of the quasi-projective strong Lefschetz property is the following.

    \begin{cor}
        The refined Hodge--Poincar\'{e} polynomials $ \underline{\H}^{i}_{U_{r,n}}(x) $ have unimodal
        coefficients. Equivalently, the Betti diagram for the Hodge numbers of $ U_{r,n} $ has unimodal columns.
    \end{cor}

    In fact, we conjecture that a stronger statement is true.

    \begin{conj}
        The refined Hodge--Poincar\'{e} polynomials $\underline{\H}^{i}_{U_{r,n}}(x) $ are real-rooted.
    \end{conj}

    Real-rootedness is known for the Chow polynomials by \cite[Theorem 1.1]{BV}.
    As evidence for the other refined Hodge--Poincar\'{e} polynomials, we have verified each $ \underline{\H}^{i}_{U_{r,n}}(x) $
    has ultra-log-concave coefficients for $ n \leq 230 $ using  Proposition \ref{prop:recursiveRHP} and a short Python script.
			
		\printbibliography
\end{document}